\documentclass[12pt]{amsart}

\usepackage[margin=1.3in]{geometry}
\usepackage{amssymb,amsmath}
\usepackage{amsfonts}
\usepackage{amsthm}
\usepackage{hyperref, xcolor}
\usepackage{mathrsfs}
\usepackage{indentfirst}
\usepackage{enumitem}

\newtheorem{thm}{Theorem}[section]
\newtheorem{lem}[thm]{Lemma}
\newtheorem{prop}[thm]{Proposition}
\newtheorem{defn}[thm]{Definition}

\newtheorem{cor}[thm]{Corollary}
\newtheorem{rem}[thm]{Remark}

\newtheorem*{clm}{Claim}
\numberwithin{equation}{section}

\def\XXint#1#2#3{{\setbox0=\hbox{$#1{#2#3}{\int}$}
		\vcenter{\hbox{$#2#3$}}\kern-.5\wd0}}

\begin{document}

	\title{Lagrangian Mean Curvature Equations on exterior domains}
	
\author[J.-G.~Bao]{Jiguang Bao}
\thanks{The first author is supported by the National Natural Science Foundation of China (12371200) and Beijing Natural
	Science Foundation (1254049).}
\author[Q.-F.~Jiang]{Qinfeng Jiang*}
\thanks{*Corresponding author.}


	\maketitle
	
	\begin{center}
	\normalsize
	School of Mathematical Sciences, Beijing Normal University,\\
	Beijing, 100875, China\\[0.3em]
	\text{jgbao@bnu.edu.cn}\\
	\text{202531130031@mail.bnu.edu.cn}
\end{center}

	\begin{abstract}
	We  introduce an extended exterior $(K,K^{\prime},\alpha_0)$--quasiconformal mapping method to study the asymptotic behavior at infinity of solutions to the supercritical phase Lagrangian mean curvature equation
	\[
	\sum_{i=1}^{n} \arctan \lambda_i(D^2u) = \theta + f(x)
	\]
	on exterior domains in $\mathbb{R}^n$, where the constant $|\theta|\in((n-2)\pi/2,n\pi/2)$, $n\geq 2$, and $f=O(|x|^{-\beta})$ is a perturbation term with the sharp decay condition $\beta>2$ at infinity. Our work generalizes the classical exterior Bernstein-type theorem for the special Lagrangian equation ($f\equiv0$) established by Li--Li--Yuan [Adv. Math. (2020)].  
	Via Perron's method, we solve the corresponding  Dirichlet problem outside a  bounded, uniformly convex domain, prescribing asymptotic behavior at infinity. For $n \geq 3$, we establish existence and uniqueness of viscosity solutions in both the supercritical phase case with $f \not\equiv 0$ and the subcritical phase case with $f \equiv 0$. This extends earlier work by Li [Trans. Amer. Math. Soc. (2019)] on the exterior Dirichlet problem for the special Lagrangian equation ($f \equiv 0$) under weaker regularity assumptions on the interior boundary and boundary data.
	\end{abstract}
	
\tableofcontents  
	
\section{Introduction}\label{sec1}
Within the framework of calibrated geometry, Harvey and Lawson \cite{HL1982} first introduced the special Lagrangian equation 
  \begin{equation}\label{eq1}
  	F(D^2u) := \sum_{i=1}^{n} \arctan \lambda_i(D^2u) = \theta,  
  \end{equation}
  where  $\lambda_i(D^2u)$ denote the eigenvalues of the Hessian $D^2u$, $i=1,2,\cdots,n$, and  $\theta \in \left(-n\pi/2,n\pi/2\right)$ is a constant. 
The left hand side of the equation \eqref{eq1} indeed stands for the argument of the complex number $(1+\sqrt{-1}\lambda_{1}(D^{2}u))\cdots(1+$ $\sqrt{-1}\lambda_{n}(D^{2}u))$, which is usually called Lagrangian phase. Its solutions $u$ were shown to have the property that the gradient graph $(x, Du(x))$
in Euclidean space is a Lagrangian submanifold which is absolutely volume-minimizing.  

The arctangent operator
is clearly elliptic for any function $u$. Because $\arctan$ is an odd function, we may assume without loss of generality that $\theta \geq 0$. In the literature \cite{Yuan2006}, the Lagrangian phase $(n - 2)\pi/2$ is usually called critical, since the level set
\[L_\theta:=\left\{\lambda=\left(\lambda_1,\cdots,\lambda_n\right)\in\mathbb{R}^n\vert\sum_{i=1}^n\arctan\lambda_i=\theta\right\}\]
is convex only if $\theta\geq (n - 2)\pi/2$.  For $\theta \ge (n-1)\pi/2$, convexity of $L_{\theta}$ is immediate, because $\lambda_i \geq 0$ for all $i$ and the operator $F$ becomes concave. For supercritical phases $\theta > (n-2)\pi/2 $, the operator $F$ can be extended to a concave operator \cite{CPW2017,CW2019}. This structural property plays a fundamental role in the analysis of solutions, particularly in establishing regularity results.

There are some rigidity theorems for special Lagrangian
equations \eqref{eq1} on the whole space, Jost--Xin \cite{JX2002} employed harmonic maps into convex subsets of Grassmannians, while Yuan \cite{Y2002} developed techniques based on geometric measure theory, both proving that any entire smooth convex solution must be quadratic polynomial, with additional constrained conditions appearing in \cite{WY2008,D2023}.  Further, Chen--Warren--Yuan's \cite{CWY2009} interior regularity theory showed that all convex viscosity solutions are smooth when $\theta \geq (n-1)\pi/2$, with subsequent work \cite{CSY2023} establishing real analyticity for all convex viscosity solutions. In \cite{Y2002,Yuan2006}, Yuan established that any entire solutions with supercritical phases  or with semiconvexity are quadratic polynomials.   The supercritical phase is a necessary condition,
as demonstrated by an entire solution $u(x)=\left(x_1^2+x_2^2-1\right)e^{-x_3}+e^{x_3}/4$ to \eqref{eq1} with
$\theta = \pi/2$ in $\mathbb{R}^3$ by Warren \cite{W2010}.

    Let $u$ be a smooth solution $u$ of \eqref{eq1} in $\mathbb{R}^n\setminus \overline{B}_1$ with supercritical phase or with semiconvexity. Li--Li--Yuan \cite{LLY2020} established
    asymptotic expansions for solutions near infinity. Specifically,  there exist some $n \times n$ symmetric matrix $A$, some vector $ b \in \mathbb{R}^n$, and some constant $c \in \mathbb{R}$, such that
\begin{equation}\label{fd4}
u(x)= \frac{1}{2}x^\mathsf{T} A x + b^\mathsf{T} x + c+O_k(|x|^{2-n})\quad \text{as} \ |x|\to +\infty,
\end{equation}
holds for dimensions $n\geq3$.   For $n=2$, there exist some $2 \times 2$ symmetric matrix $A$, some vector $ b \in \mathbb{R}^n$, and some constant $c,d \in \mathbb{R}$,
\begin{equation}\label{fd5}
	u(x)= \frac{1}{2}x^\mathsf{T} A x + b^\mathsf{T} x +\frac{d}{2}\log\left(x^{\mathsf{T}}(I+A^2)x\right) + c+O_k(|x|^{-1})\quad \text{as} \ |x|\to +\infty,
\end{equation}
for all $k\in \mathbb{N}$, the notation $\phi(x)=O_k\left(|x|^{k_1}(\ln |x|)^{k_2}\right)$, $k_1\leq 0, k_2\geq 0$, means that $|D^t\phi(x)|=O(|x|^{k_1-t}(\ln |x|)^{k_2})$ for $t=0,\cdots,k$.

Subsequent works by Liu--Bao \cite{LB2022JGA,LB2023} derived higher order asymptotic expansions. Recently, Han--Marchenko\cite{HM2025}  also used a single function $v$ to characterize remainders in the assymptotic expansions via a modified Kelvin transform. When the right-hand side of \eqref{eq1} is perturbed as $\theta + f(x)$ with $f(x) \to 0$ at infinity, Liu--Bao \cite{LB2022} established the asymptotic expansion at infinity in dimension $2$ under the assumption of quadratic growth. 

In this paper,  our first objective is to study   the asymptotic behaviors at infinity of solutions for Lagrangian mean curvature equations in exterior domian.

Let $\Omega$ be a bounded domain in $\mathbb{R}^n$ ($n \geq 2$), consider  Lagrangian mean curvature equation
\begin{equation}\label{eq}
	F(D^2u) := \sum_{i=1}^{n} \arctan \lambda_i(D^2u) = \theta + f(x) \quad \text{in } \mathbb{R}^n \setminus \overline{\Omega}.
\end{equation}
 Throughout this paper, we assume that
$f$ is $C^m$ near infinity, and 
\begin{equation}\label{abf}
	f(x) = O_m(|x|^{-\beta}) \quad \text{as } |x| \to \infty,
\end{equation}
where $m\geq2$ and    $\beta>2$.

 In contrast to the Monge-Amp\`{e}re equation, solutions to \eqref{eq} do not satisfy a subquadratic growth estimate for $u$ minus a quadratic polynomial \cite{CL2003,BLZ2015}. The principal difficulty is establishing the existence of the limit $A$ of $D^2u$ at infinity, as noted in \cite[Theorem 2.1]{LLY2020}. In dimension $2$, departing from the barrier function techniques employed by Li--Li--Yuan \cite{LLY2020} and Jia \cite{J2020}, we develop in Section~\ref{sec2.3} an extended exterior $(K,K^{\prime},\alpha_0)$--quasiconformal mapping method to obtain the  decay rate $|D^2 u(x) - A|$ as $|x| \to +\infty$. For standard treatments of $K$--quasiconformal and $(K,K')$--quasiconformal mappings, we refer the readers to \cite{GT}. 
 The priori interior H\"{o}lder estimate is well known (see, e.g., \cite[Lemma 2]{N53} and \cite[Theorem 1]{FS58}).

Recall again, without loss of generality, we assume $\theta\geq 0$. Let
\[\mathcal{A}:=\{M \in \mathcal{S}^{n\times n}|  \sum_{i=1}^{n} \arctan\lambda_i(M)=\theta   \},\]
where $\mathcal{S}^{n\times n}$ denotes the linear space of symmetric $n\times n$ real matrices.

The first two main results of this paper, which address the asymptotic behaviors of solutions at infinity, are given below for the cases $n=2$ and $n\geq3$  respectively.
\begin{thm}\label{D2ue}
Let $n=2$, $\theta>0$, and suppose $u$ is a smooth solution of~\eqref{eq}. Assume that $f$ satisfies condition~\eqref{abf} with constants $\beta > 2$ and $m \geq 3$. If there exist  positive constants $C_0$ and $R_0$ such that
\begin{equation}\label{Ducond}
	|Du(x)| \, \bigl| D^{l} f(x) \bigr|^{\frac{1}{\beta+l}} \leq C_0, \qquad \forall |x| \geq R_0,
\end{equation}
 for $l = 0,1,2,3$. Then there exist a  matrix $A \in \mathcal{A}$, a vector $b \in \mathbb{R}^2$, and constants $c, d \in \mathbb{R}$ such that
	\begin{equation}\label{bg2}
		u(x)=\frac{1}{2}x^\mathsf{T} Ax+b^\mathsf{T} x+\frac{d}{2}\ln\left(x^{\mathsf{T}}(I+A^2)x\right)+c+O_k\left(|x|^{-\min \left\{1,\beta-2 \right\} } (\ln |x|)^{\mu_1}\right)
	\end{equation}
	as $|x|\to+\infty$, for all $k=0,\cdots,m+1$. Here, $\mu_1 = 
	\begin{cases} 
		0, & \beta \neq 3, \\ 
		1, & \beta = 3.
	\end{cases}$.
	
	Moreover, if $\beta>3$, there exists  vector $ e\in \mathbb{R} ^2$, such that  for  some $ \zeta \in (1,2)$ depending only on $\beta$, such that
	\begin{equation}\label{bg3}
				u(x)=\frac{1}{2}x^\mathsf{T} Ax+b^\mathsf{T} x+\frac{d}{2}\ln\left(x^{\mathsf{T}}(I+A^2)x\right)+c+\frac{e^\mathsf{T} x}{x^{\mathsf{T}}(I+A^2)x}+O_k\left(|x|^{-\zeta}\right)
	\end{equation}
	as $|x|\to+\infty$, for all $k=0,\cdots,m+1$.	$d$ in \eqref{bg2} and \eqref{bg3} is given by
	\begin{align*}
		d=\frac{ \sqrt{\det(I+A^2)}}{2\pi}
	 \cdot\int_{\mathbb{R}^2}\left(\left(I+A^2\right)_{ij}D_{ij}u(x)-\mathrm{tr}A\right)\mathrm{d}x.
	\end{align*}
\end{thm}

\begin{thm}\label{ng3r}
Let $n\geq3$, $\theta>(n-2)\pi/2$, and suppose $u$ be a smooth solution of \eqref{eq}. Assume that $f$ satisfies \eqref{abf}
for	$m\geq2$ and $\beta>2$. If  there exist positive constants  $C_0$ and  $R_0$ such that 
	\begin{enumerate}
		\item[(i)] $|Du(x)||D^lf(x)|^{\frac{1}{\beta+l}}\leq C_0$ holds for $ |x|\geq R_0$ and $l=0,1,2$. 
		\item[(ii)]  $\beta$ is sufficiently large.
	\end{enumerate}
	Then there exist a  matrix $ A\in \mathcal{A}$, $b\in \mathbb{R} ^n, c\in \mathbb{R}$ 
	such that
	\[u(x)=\frac12x^\mathsf{T} Ax+b^\mathsf{T} x+c+O_{k}\left(|x|^{2-\min \left\{\beta,n \right\}} (\ln |x|)^{\mu_2}\right)\]
	as $|x|\to +\infty$, for all $k=0,\cdots,m+1$. Here,  $\mu_2 = 
	\begin{cases} 
		0, & \beta \neq n, \\ 
		1, & \beta = n.
	\end{cases}$. 
\end{thm}

\begin{rem}
	Note that in Theorem \ref{D2ue}, we assumed that $f$ possesses at least $C^3$ regularity, a condition that is stronger than that required in Theorem \ref{ng3r}. In contrast to
	higher--dimensional settings, in two dimensions we employ the iteration method developed in \cite{BLZ2015}, rather than appealing to the equivalence property of Green's functions on unbounded domains; see \cite{LB2021,P1992}.
\end{rem}

 Our assumaptions in Theorems \ref{D2ue} and \ref{ng3r} are clearly satisfied if $f\equiv 0$. Consequently, we have the following corollary.
\begin{cor}
	Let $u$ be a smooth solution of the special Lagrangian equation
	\[\sum_{i=1}^{n} \arctan \lambda_i(D^2u) = \theta  \quad \text{in } \mathbb{R}^n \setminus \overline{\Omega},\]
	where $\theta>(n-2)\pi/2$, $\Omega$ is a bounded set.
	There exist a  matrix $ A\in \mathcal{A}$, $b\in \mathbb{R}^2, c\in \mathbb{R}$ such that \eqref{fd4} or \eqref{fd5} holds, and $d$ given as above.
\end{cor}

\begin{rem}
	The above Corollary is a classical result proved by Li--Li--Yuan \cite[Theorem 1.1]{LLY2020}. Therefore, Theorems \ref{D2ue} and \ref{ng3r} can be seen as a extension of their Theorem under the supercritical phases. In \cite{LB2022}, the authers also consider the asymptotic behavior of \eqref{eq} for $n=2$ with supercritical phases and  quadratic growth at infinity assumption.  By interior gradient estimate \cite{BMS2022}, the gradient of solution for  Eq. \eqref{eq} with quadratic growth at infinity implies linear growth at infinity, which clearly satisfies condition~\eqref{Ducond}. For Liouville type theorem of nonlinear second elliptic equations,  such growth condtions can be also found in \cite{BCGJ,WY2008,WB2022}.
\end{rem}

\begin{rem}\label{sprem}
	The following example shows that the decay rate assumption $\beta > 2$   is sharp in Theorems  \ref{D2ue} and \ref{ng3r}.
	
	With out loss generality, we assume $B_1\subset \Omega $. Similar to the construction in \cite[Pages 27--33]{BLW2024},
	we may construct radially symmetric classical solution  $u$ to  \eqref{eq} with right hand term $\theta+|x|^{-2}$  has asympototic behavior
	\[	u_0(x)=\frac{1}{2}\tan \frac{\theta }{2}|x|^{2}+	O((\ln|x|)^{2}),\quad n=2;\]
	and
	\[	u_0(x)=\frac{1}{2}\tan \frac{\theta }{2}|x|^{2}+	O(\ln|x|),\quad n\geq 3,\]
		as $ |x|\to \infty$. 
\end{rem}

Our second objective is to study  the exterior Dirichlet problem for Lagrangian mean curvature equations. 

For the defination of viscosity solutions, we refer to \cite{CIL1992,IL1990,JLS1988}.

  In  \cite{LZ2019},  Li investigated the exterior problem for the special Lagrangian equation ($f\equiv0$). Specifically, for dimensions $n \geq 3$, given a strictly convex $C^2$ boundary domain $\Omega$ and boundary data $\varphi \in C^2(\partial \Omega)$, he proved that for any given symmetric positive definite  matrices $A$ satisfying condition \cite[(1.6)]{LZ2019} with $(n-2)\pi/2 \leq \theta < n\pi/2$,  and any given vector $b \in \mathbb{R}^n$,    there exists a  constant $c_*$ depending on $n$, $\Omega$, $\theta$, $A$, $b$ and $\|\varphi\|_{C^2(\partial \Omega)}$, such that for every $c>c_*$, the exterior Dirichlet problem admits a unique viscosity solution $u\in C^0\left(\mathbb{R}^n\setminus\Omega\right)$ of 
  \begin{equation}\label{Ledeq}
  	\begin{cases}
  		\sum_{i=1}^n \arctan\lambda_i(D^2u) = \theta \quad \text{in } \mathbb{R}^n \setminus \overline{\Omega}, \\
  		u = \varphi \quad \text{on } \partial \Omega,\\
  	 u(x)=\frac{1}{2}x^\mathsf{T} Ax + b^\mathsf{T} x + c+o(1),\quad \text{as} \ |x|\to +\infty. 
  	\end{cases}
  \end{equation}

Take note that the above result was established under the critical or supercritical assumption. A natural question is whether we can extend it to subcritical case. Moreover,  the right hand of equation of the first line in \eqref{Ledeq} is a constant,  it would be interesting to investigate whether the phase $\theta$ can be generalized to a perturbed form $\theta + f(x)$, where $f(x) \to 0$ at infinity.


Define
\[\mathcal{A}_{0}:=\left\{A\in \mathcal{A}\vert \ A \ \text{is positive definite and} \  d(A)>2 \right\}\]
where 
	\begin{equation}\label{MAeps}
d(A) = \frac{1+\lambda_{min}^2(A)}{\lambda_{max}(A)} \cdot \sum_{i=1}^n\frac{\lambda_{i}(A)}{1+\lambda_{i}^2(A)}.
\end{equation}

$d(A)$ plays an important role
in the construction of the subsolution and supersolution in   Lagrangian mean curvature equations.  To a certain extent,  $d(A)>2$  implies that the maximum eigenvalue of \(A\) is controlled by its minimum eigenvalue. This condition is trivially satisfied when \(A =A_{*} I\), where $A_{*}=\tan (\theta/n)$ and $I$ is  the identity matrix.  Although similar assumptions have appeared in the context of exterior Dirichlet problems \cite{BLW2024,LW2024,LZ2019},  whether this assumption can be relaxed or removed remains an open question.

In order to  describe the solvability of the exterior Dirichlet problem under weaker assumptions on the interior boundary and its boundary data, we employ the concept of semi-convexity. For a detailed discussion, we refer the reader to \cite{BW2024,WB2024arxiv}.

Our second main result of this paper are as follows.
\begin{thm}\label{tedp}
	 Let \( \Omega \) be a bounded, strictly convex domain in \( \mathbb{R}^n \), \( n \geq 3 \), \( \partial \Omega \in C^{1,1} \). Let \( \varphi \) be semi-convex with respect to $\partial \Omega$. Then for any given \( A \in \mathcal{A}_{0} \) and \( b \in \mathbb{R}^n \), 
	 \begin{itemize}
	 	\item[(i)] if  $(n-2)\pi/2<\theta<n\pi/2$ and $f$ satisfies \eqref{abf}, there exists a constant \( c_\ast \) depending only on \( n, \Omega, \theta, A, b \), \(\varphi\) and $f$, such that for every \( c \geq c_\ast \), there exists a unique viscosity solution \( u \in C^0(\mathbb{R}^n \setminus \Omega) \) of
	 	\begin{equation}\label{eedp}
	 		\left\{ \begin{array}{ll}
	 			\sum_{i=1}^n \arctan \lambda_i (D^2u) = \theta+f(x) \quad  \text{in }\ \mathbb{R}^n \setminus \overline{\Omega},\\
	 			u = \varphi \quad  \text{on }\ \partial \Omega,\\
	 			u(x)=\frac{1}{2}x^\mathsf{T} Ax + b^\mathsf{T} x + c+O(|x|^{2-\min\{\beta,d(A)\}}),\quad \text{as} \ |x|\to +\infty, 
	 		\end{array}
	 		\right.
	 	\end{equation}
	 for $\beta\neq d(A)$.	The asymptotic behavior is replaced by
	 	\[u(x)=\frac{1}{2}x^\mathsf{T} Ax + b^\mathsf{T} x + c+O(|x|^{2-d(A)}\ln|x|),\quad \text{as} \ |x|\to +\infty,\]
	 	for $\beta= d(A)$.
	 	\item[(ii)]  If  $0<\theta< n\pi/2$  and $f\equiv 0$, there exists a constant \( c^\ast \) depending only on \( n, \Omega, \theta, A, b \) and \(\varphi\), such that for every \( c \geq c^\ast \), 
	 	there exists a unique viscosity solution \( u \in C^0(\mathbb{R}^n \setminus \Omega) \) of
	 	\begin{equation}\label{eedp1}
	 		\left\{ \begin{array}{ll}
	 			\sum_{i=1}^n \arctan \lambda_i (D^2u) = \theta \quad  \text{in }\ \mathbb{R}^n \setminus \overline{\Omega},\\
	 			u = \varphi \quad  \text{on }\ \partial \Omega,\\
	 			u(x)=\frac{1}{2}x^\mathsf{T} Ax + b^\mathsf{T} x + c+O(|x|^{2-d(A)+\epsilon_0}),\quad \text{as} \ |x|\to +\infty, 
	 		\end{array}
	 		\right.
	 	\end{equation}
	 	for some positive constant  $\epsilon_0<d(A)-2$.
	 \end{itemize}
\end{thm}
\begin{rem}
	In problem \eqref{eedp}, if we replace the right hand of the first line equation with any function lying in the
	subcritical or critical range, then the existence and uniqueness of $C^0$ viscosity solutions remain open. Since the locally solvablity of the Dirichlet problem for Lagrangian mean curvature equations with dimensions great than $3$  are only known in two cases: for supercritical phases when the perturbation $f \not\equiv  0$, and for all phases $\theta \in (-n\pi/2, n\pi/2)$ when $f \equiv 0$, as established in \cite[Theorems~1.1 and 1.2]{AB2024}.
\end{rem}

\begin{rem}
  The prescribed asymptotic behavior in problem \eqref{eedp} is optimal, but in \eqref{eedp1} is not optimal. Indeed, for $0<\theta< n\pi/2$ and $A = A_* I$, it is clear that $d(A_* I) = n > 2$, yet for any given $b\in \mathbb{R}$ and large $c$, the following problem
 \begin{equation*}
 	\left\{
 	\begin{array}{l}
 		 \sum_{i=1}^n \arctan \lambda_i(D^2u) = \theta+|x|^{-\beta} \quad \text{in } \mathbb{R}^n \setminus \overline{B}_1, \\
 		u = b \quad \text{on } \partial B_1, \\
 		u(x) = \frac{A_*}{2}|x|^2 + c + O(|x|^{2-\min \left\{\beta,n\right\}}), \quad |x| \to +\infty,
 	\end{array}
 	\right.
 \end{equation*}
 admits a unique radially symmetric solution of the form
 \begin{equation*}
 	u(x) = b + \int_1^{|x|} \bigl( w(\overline{c}\tau^{-\min \left\{\beta,n\right\}}) - w(0) \bigr) \cdot \tau \, d\tau,
 \end{equation*}
 where $w$ is a function and $\overline{c}$ is a constant defined in \cite[Page~18]{LB2022}.
\end{rem}

\begin{rem}
 For $n\geq 3$,	Theorem \ref{tedp} gives a  positive answer to Remark 1.1(3) in  \cite{LZ2019}.  Our method, however, does not yield existence for the exterior Dirichlet problem in dimension two. The principal difficulty is constructing the suitable sub and supersolutions  satisfying  asymptotic behavior \eqref{bg2} at infinity. 
\end{rem}

This paper is organized as follows. In Section~\ref{sec2}, we introduce some standard notations, followed by the definition of the Lewy rotation in Section~\ref{sec2.2}. In Section~\ref{sec2.3}, we intordece the defination of extended exterior $(K,K^{\prime},\alpha_0)$--quasiconformal mapping and  analyze the asymptotic behavior of the Hessian at infinity. We also collect several useful lemmas for later use in Section~\ref{sec2.4}. Section~\ref{sec3} is devoted to the proofs of Theorems~\ref{D2ue} and~\ref{ng3r}. In Section~\ref{sec4}, we construct generalized symmetric subsolutions and supersolutions of the Lagrangian mean curvature equation. Finally, Theorem~\ref{tedp} is proved via Perron's method in Section~\ref{sec5}. 

\section{Preliminaries}\label{sec2}
\subsection{Notation}\label{sec2.1}

For any \( M \in \mathcal{S}^{n \times n}  \), if \( m_1, m_2, \cdots, m_n \) are the eigenvalues of \( M \) (usually, the assumption \( m_1 \leq m_2 \leq \cdots \leq m_n \) is added for convenience), we will denote this fact briefly by \( \lambda(M) = (m_1, m_2, \cdots, m_n) \) and call \( \lambda(M) \) the eigenvalue vector of \( M \). For convenience, we write often \(\lambda(I) = (1, 1, \cdots, 1) \).

Denote 
\[
\mathcal{M}_+ = \{ M\in \mathcal{S}^{n \times n} \mid \lambda(M)>0 \}.
\] 
  For \(A \in \mathcal{M}_+\) and \( \rho > 0 \), we denote by
\[
E_\rho :=  \{ x \in \mathbb{R}^n \mid r_A(x) < \rho \}
\]
the ellipsoid  with respect to \( A \), where we set \( r_A(x) := \sqrt{x^\mathsf{T} A x} \).


 If \(\Phi(x) := \phi(r)\) with \(\phi \in C^2\), \(r_{A}(x)= \sqrt{\sum_{i=1}^n a_ix_i^2}\) and $(a_1, a_2, \cdots, a_n)=\lambda(A)$, 
 we may call \(\Phi\) a generalized radially symmetric function with respect to \(A\in \mathcal{M}_+\) (see \cite{BLL2014} for more details), a direction computation yields
 \begin{equation}
 	\label{gra}
 	D_{ij}\Phi(x) = a_ih\delta_{ij} + r^{-1}h^{'}(a_i x_i)(a_j x_j), \quad \forall 1 \leq i, j \leq n,
 \end{equation}
where $h(r)=\phi'(r)/r$.

\subsection{Lewy rotation in supercritical phase}\label{sec2.2}

  Throughout this paper, we may assume that $\theta+f(x)> (n-2)\pi/2+2\delta$,  $|f(x)|\leq \delta$ for some positive constant $\delta$ and $|x|$ large.
  
  As in \cite{LLY2020,Y2002,Yuan2006},  we first make a transformation of the solution, so that the Hessian of the new potential $\tilde{u}$ is bounded.  Selecting the rotation angle $\vartheta=\delta/n$, let 
\begin{equation}\label{rtfm}
	\tilde{x} = \mathfrak{c}x + \mathfrak{s} Du(x), \quad \tilde{y} = -\mathfrak{s}x + \mathfrak{c} Du(x), 
\end{equation}
where \((\mathfrak{c},\mathfrak{s}) = (\cos\vartheta, \sin\vartheta)\).

 Defined $\tilde{u}(\tilde{x})=\int\limits^{\tilde{x}}\tilde{y} d\tilde{x}$.  Then   $D_{\tilde{x}}\tilde{u}=\tilde{y}$,
 \[D_{\tilde{x}}^2\tilde{u}=\left(-\mathfrak{s}I+\mathfrak{c} D^2u\right)\left(\mathfrak{c}I+\mathfrak{s}D^2u\right)^{-1},\]
 and $\tilde{u}$ satisfies the equation
 \begin{equation}\label{npeq}
  \sum_{i=1}^n \arctan \lambda_i (D_{\tilde{x}}^2 \tilde{u}) = \tilde{\theta} + \tilde{f} (\tilde{x}, D_{\tilde{x}}\tilde{u} (\tilde{x})) \quad \text{and} \quad |D_{\tilde{x}}^2 \tilde{u}| < \cot \vartheta \quad \text{in} \ \mathbb{R}^n \setminus \overline{\tilde{\Omega}},
 \end{equation}
where $\tilde{\theta}=\theta - \delta>(n-2)\pi/2$,   \( \tilde{f} (\tilde{x}, D_{\tilde{x}}\tilde{u} (\tilde{x})) =f(x)= f(\mathfrak{c}\tilde{x} - \mathfrak{s} D_{\tilde{x}} \tilde{u} (\tilde{x})) \) and $\tilde{\Omega}=\tilde{x}(\Omega)$ is a bounded domain.

Consequently, the right-hand side of \eqref{npeq} lies in the supercritical phase, which ensures that \(F\) is concave in the level set sense and can be modified into a concave operator.  For further details, see \cite[Lemma 2.2]{CPW2017} and \cite{CW2019}. Without ambiguity, we continue to denote the modified operator by $F$ and regard it as concave.

 For the new potential  Eq. \eqref{npeq}, we have the folowing propersition.
 \begin{prop}\label{rtp} 
  Let $l(\tilde{x})=\tilde{f} (\tilde{x}, D_{\tilde{x}}\tilde{u} (\tilde{x}))$.	Suppose the potential function satisfies \eqref{rtfm} and \eqref{npeq},  then
 	\begin{itemize}[label=(\roman*)]
 		\item[(i)]  $|x|\to +\infty$ as $|\tilde{x}|\to +\infty$.
 		\item[(ii)] 
 		  $|D_{\tilde{x}}l|\leq C(n,\delta)|D_{x}f|$.
 		\item[(iii)] 
 	 $|D_{\tilde{x}}^2l|\leq C(n,\delta)\left(|D_{x}f||D_{\tilde{x}}^3\tilde{u}|+|D_{x}^2f|\right)$.
 		\item[(iv)]   
 		        $|D_{\tilde{x}}^3l|\leq C(n,\delta)\left(|D_{x}f||D_{\tilde{x}}^4\tilde{u}|+|D_{x}^2f||D_{\tilde{x}}^3\tilde{u}|+|D_{x}^3f|\right)$.    
 		\item[(v)] If there exist some constant symmetric matrix $\tilde{A}$ satisfying $F(\tilde{A})=\tilde{\theta}$, $\tilde{b}\in \mathbb{R}^n$ and $\xi>0$ such that
 		\[D_{\tilde{x}}\tilde{u}(\tilde{x})=\tilde{A}\tilde{x}+\tilde{b}+O_1(|\tilde{x}|^{-\xi})\quad \text{as}\ |\tilde{x}|\to +\infty.\]
 		Then 
 		 $D_x^2u(x)\to A$ as $|x|\to\infty$ with
 		\[A=\left(\mathfrak{s}I+\mathfrak{c}\tilde{A}\right)\left(\mathfrak{c}I-\mathfrak{s}\tilde{A}\right)^{-1}.\]
 	\end{itemize}
 \end{prop}
 \begin{proof}
Since the map $x \mapsto \tilde{x}$ is a diffeomorphism from $\mathbb{R}^n$ onto itself (see \cite{LLY2020}), (i) is obviously.

 (ii)--(iv) can be shown via a direct computation.   (v) is an arguement in \cite[Section 3.1]{LLY2020}.
 \end{proof}
 
 By the asymptotic behavior of \eqref{abf} and Eq. \eqref{npeq}, we have
 \begin{lem}\label{abf11}
Suppose  $f$ satisfies \eqref{abf} and \eqref{Ducond}. If the potential function $\tilde{u}$ satisfies \eqref{rtfm} and  \eqref{npeq},  then
 	\begin{equation}
 		l(\tilde{x})=O_{k}\left(|\tilde{x}|^{-\beta}\right) \quad \text{as} \ |\tilde{x}|\to \infty,\quad k=0,1,2,3.
 	\end{equation}
 \end{lem}
\begin{proof}
	In view of \eqref{abf},  there exists large $R_1\geq R_0$ and constant $C_1$ such that
	\[|x|^{\beta+k}|D_x^kf(x)|\leq C_1,\quad \forall |x|\geq R_1.\]
	Therefore, by \eqref{Ducond}, \eqref{rtfm} and Propsition \ref{rtp}(i), there exists large $R_2$ such that for $|\tilde{x}|\geq R_2$, we have $|x|\geq R_1$ and 
	\[|l(\tilde{x})||\tilde{x}|^{\beta}\leq C(n,\beta,\delta)\left(|x|^{\beta}+|Du(x)|^{\beta}\right)|f(x)|\leq C(n,\beta,\delta,C_0,C_1).\]
	
	Similarly,  by \eqref{Ducond}, \eqref{rtfm} and Propsition \ref{rtp}(ii), there exists large $R_2$ such that for $|\tilde{x}|\geq R_2$,
	\[|D_{\tilde{x}}l(\tilde{x})||\tilde{x}|^{\beta+1}\leq C(n,\beta,\delta)\left(|x|^{\beta+1}+|Du(x)|^{\beta+1}\right)|D_x f|\leq C(n,\beta,\delta,C_0,C_1).\]
	
	For sufficiently large $R:= |\tilde{x}|\geq R_2 \gg 1$, set
	\[\tilde{u}_{R}(\tilde{y}):=\left(\frac{4}{R}\right)^{2}\tilde{u}\left(\tilde{x}+\frac{R}{4}\tilde{y}\right)\quad \text{in} \ B_{2}.\]
	Then,
	\begin{equation}\label{eqR}
		\begin{aligned}
				F(D_{\tilde{y}}^{2}\tilde{u}_{R}(\tilde{y}))=F\left(D_{\tilde{x}}^{2}\tilde{u}\left({\tilde{x}}+\frac{R}{4}\tilde{y}\right)\right)&=\tilde{\theta}+\tilde{f}\left({\tilde{x}}+\frac{R}{4}\tilde{y},D_{\tilde{x}}\tilde{u}\left({\tilde{x}}+\frac{R}{4}\tilde{y}\right)\right)\\
				&=\tilde{\theta}+\tilde{f}\left({\tilde{x}}+\frac{R}{4}\tilde{y},\frac{R}{4}D_{\tilde{y}}\tilde{u}_{R}\left(\tilde{y}\right)\right)\\
				&=:\hat{f}(\tilde{y},D_{\tilde{y}}\tilde{u}_{R}(\tilde{y})).
		\end{aligned}
	\end{equation}
	
	For $\tilde{p}=D_{\tilde{y}}\tilde{u}_{R}$, by \eqref{abf}, \eqref{Ducond} and  \eqref{npeq}, it's easy to see that
	\[\frac{\partial \hat{f}}{\partial \tilde{y}_i}=\frac{R\mathfrak{c}}{4}f_i, \quad  \frac{\partial \hat{f}}{\partial \tilde{p}_i}=-\frac{R\mathfrak{s}}{4}f_i, \ i=1,\cdots,n,\]
	and 
	\[\frac{\partial^2 \hat{f}}{\partial \tilde{y}_i\partial \tilde{y}_j}=\frac{R^2\mathfrak{c}^2}{16}f_{ij},\quad \frac{\partial^2 \hat{f}}{\partial \tilde{y}_i\partial \tilde{p}_j}=-\frac{R^2\mathfrak{c}\mathfrak{s}}{16}f_{ij},\quad \frac{\partial^2 \hat{f}}{\partial \tilde{p}_i\partial \tilde{p}_j}=\frac{R^2\mathfrak{s}^2}{16}f_{ij},\ i,j=1,\cdots,n, \] 
	  is bounded by $C_0$ and $C_1$.
	
	Note that  $\|\tilde{u}_{R}\|_{C^2(B_2)}$ is bounded by \eqref{npeq}, and  \( F \) is concave. Then by  by the Evans--Krylov estimates and the Schauder theory in \cite{GT},
	\[|D_{\tilde{y}}^{3}\tilde{u}_{R}(\tilde{y})|\leq C(n,\delta,C_1),\quad \tilde{y}\in B_{\frac{1}{2}}.\]
	
	Combine the \eqref{Ducond}, \eqref{rtfm} and Propsition \ref{rtp}(iii),  there exists large $R_3$ such that for $|\tilde{x}|\geq R_3$, $|D_{\tilde{x}}^{3}\tilde{u}(\tilde{x})||\tilde{x}|$ is bounded, and  
	\begin{align*}
		|D_{\tilde{x}}^2l(\tilde{x})||\tilde{x}|^{\beta+2}&\leq C(n,\beta,\delta)\left(|x|^{\beta+2}+|Du(x)|^{\beta+2}\right)|D_x^2f|\\
		&\quad +C(n,\beta,\delta)|\tilde{x}|^{\beta+1}|D_xf||\tilde{x}||D_{\tilde{x}}^{3}\tilde{u}(\tilde{x})|\\
		&\leq  C(n,\beta,\delta,C_0,C_1).
	\end{align*}
	
	Similarly, 
	$|D_{\tilde{y}}^{4}\tilde{u}_{R}(\tilde{y})|\leq C(n,\delta,C_1)$, 
$ \tilde{y}\in B_{\frac{1}{4}}$,  and there exists large $R_4$ such that for $|\tilde{x}|\geq R_4$, $|D_{\tilde{x}}^{4}\tilde{u}(\tilde{x})||\tilde{x}|^2$ is bounded.

Therefore,  \eqref{Ducond}, \eqref{rtfm} and Propsition \ref{rtp}(iv) imply
  	\begin{align*}
  	|D_{\tilde{x}}^3l(\tilde{x})||\tilde{x}|^{\beta+3}&\leq C(n,\beta,\delta)\left(|x|^{\beta+3}+|Du(x)|^{\beta+3}\right)|D_x^3f|\\
  	&\quad +C(n,\beta,\delta)|\tilde{x}|^{\beta+2}|D_x^2f||\tilde{x}||D_{\tilde{x}}^{3}\tilde{u}(\tilde{x})|\\
  	&\quad +C(n,\beta,\delta)|\tilde{x}|^{\beta+1}|D_xf||\tilde{x}|^2|D_{\tilde{x}}^{4}\tilde{u}(\tilde{x})|\\
  	&\leq  C(n,\beta,\delta,C_0,C_1).
  \end{align*}
	This proves the desired result.
\end{proof}

\subsection{The limit of Hessian at infinity}\label{sec2.3}

To get the limit of $D^2u$ in $\mathbb{R}^2$, quasiconformal mappings is an useful tool.  Let’s begin with the definition of   extended exterior  $(K,K^{\prime},\alpha_0)$ quasiconformal mappings in $\mathbb{R}^2\setminus \overline{\Omega}$.
\begin{defn}\label{eeq}
	A mapping $w(x)=(p(x),q(x))$ from $\mathbb{R}^2\setminus\overline{\Omega}$ ($\Omega$ is bounded) in $x=(x_1,x_2)$ plane to $w=(p,q)$ plane is extended exterior  $(K,K^{\prime},\alpha_0)$ quasiconformal, if $p,q\in$ $C^{1}\left(\mathbb{R}^{2}\setminus\overline{\Omega}\right)$ and
	\begin{equation}\label{ceeq}
		p_1^2+p_2^2+q_1^2+q_2^2\leq2K(p_1q_2-p_2q_1)+K^{\prime}|x|^{-2\alpha_0}
	\end{equation}
	holds for all $x\in\mathbb{R}^2\setminus\overline{\Omega}$ with some constants $K\geq 1$, $K^{\prime}\geq 0$ and $\alpha_0 >1$, where $p_i=\partial p/\partial x_i,q_i=\partial q/\partial x_i, i=1,2$. 
\end{defn}

For  exterior $K$-quasiconformal mappings, namely, $K^{\prime}=0$ in \eqref{ceeq}, Li--Liu \cite[Theorem 2.2]{LL24} give the H\"{o}lder estimate over
exterior domain and the asymptotic behavior at infinity. Inspired by this, we have  the following  H\"{o}lder estimate and the asymptotic behavior at infinity for our extended exterior  $(K,K^{\prime},\alpha_0)$ quasiconformal mappings.

\begin{thm}\label{abw}
	Let $w= (p, q)$  be extended exterior  $(K,K^{\prime},\alpha_0)$   quasiconformal mappings in $\mathbb{R} ^2\setminus \overline{\Omega }$ $ (\Omega \subset \mathbb{R}^2$ is bounded with
	$K> 1$, and suppose $|w| \leq M$. Then, for any $\Omega^{\prime }\supset \supset \Omega$ with $d=$dist$(\Omega,\partial\Omega^{\prime})$,
	 $w(x)$  tends to a limit $w( \infty )$ at infnity such that
	$$|w(x)-w(\infty)|\leq C|x|^{-\alpha} \quad  \text{for any} \ x\in\mathbb{R}^2 \setminus\overline{\Omega^{\prime}},$$
	where $\alpha=\min \left\{K-(K^2-1)^{\frac12}, \alpha_0-1\right\}$, $C$ depends only on $K, K^{\prime}, d$ and $M$. If $K-(K^2-1)^{\frac12}=\alpha_0-1$, the exponent $\alpha$ must be replaced by $K-(K^2-1)^{\frac12}-\epsilon$, where $\epsilon$ is an arbitrary small
	positive number.  If $K^{\prime}=0$, it's also valid for $K=1$ and then $C$ depends only on $K, d$ and $M$.
\end{thm}

To prove Theorem \ref{abw}, we need the following H\"{o}lder continuity of  $(K,K^{\prime})$  quasiconformal mappings with singularities in bounded domain.
\begin{thm}[\cite{FS58}]\label{fsthm}
	Let $w=(p,q)$ be $(K,K^{\prime})$   be a continuously differentiable mapping defned	in the domain $0< |x| \leq 1$, namely,
	\begin{equation}\label{FSthm4}
		p_1^2+p_2^2+q_1^2+q_2^2\leq 2K(p_1q_2-p_2q_1)+K^{\prime}|x|^{-2\lambda_0},
	\end{equation}
	$K$, $K^{\prime}$, and $\lambda_0$ are constants, $K>1$, $K_{1}\geq0$, and $\lambda_0<1$. Also assume
	\begin{equation}\label{pcond}
		p=o(|x|^{-\mu}) \quad \text{as} \ x\to 0,
	\end{equation}
	where $\mu=K-(K^{2}-1)^{1/2}$. Then $w$ can be defined at $x=0$ so that the resulting function is continuous in $0\leq|x|\leq1$. Moreover, for any closed domain $\Omega^{\prime}\subset \subset B_1$, $w$ satisfies
	a uniform H\"{o}lder inequality
	$$|w(x)-w(y)|\leq C|x-y|^\alpha,\ x,y\in\Omega^\prime,$$
	where $\alpha=\min \left\{\mu, 1-\lambda_0\right\}$, $C$ depends only on $K, K^{\prime}$, $d=$dist$(\Omega ^{\prime}$, $\partial \Omega)$ and $\mu$. If $\mu=1-\lambda_0$, the exponent must be replaced by $\mu-\epsilon$, where $\epsilon$ is an arbitrary small
	positive number. If $K^{\prime}=0$, it's also valid for $K=1$ and then $C$ depends only on $K$ and $d$.
\end{thm}

\begin{rem}\label{wbd}
	It's clear that  \eqref{pcond} holds when $w$ is bounded. 
\end{rem}

Next, we intrduce the  Kelvin transform of a scalar function $p(x)$ in $\mathbb{R} ^2 $, 
\[\tilde{p}(x)=p\left(\frac{x}{|x|^2}\right), \quad x\neq 0.\]
And we will show if $w$ is a extended exterior  $(K,K^{\prime},\alpha_0)$ quasiconformal, then  $\tilde{w}:=(\tilde{q},\tilde{p})$ satisfying \eqref{FSthm4}.

\begin{lem}\label{klem}
	Let $w = (p,q)$ be extended exterior $(K,K^{\prime},\alpha_0)$ quasiconformal in $\mathbb{R}^2 \setminus \overline{B}_1(0)$. Let $\tilde{p}$ and $\tilde{q}$ be the Kelvin transform of $p$ and $q$ respectively. Then, $\tilde{w} = (\tilde{q}, \tilde{p})$ satisfying \eqref{FSthm4} with $\lambda_0=2-\alpha_0$ in $B_1(0)\setminus\{0\}$.
\end{lem}
\begin{proof}
	By a direcet calcaulation, we get
	\begin{align*}
		\tilde{p}_1&=\left(|x|^{-2}-2x_1^2|x|^{-4}\right)p_1+\left(-2x_1x_2|x|^{-4}\right)p_2, \\
		\tilde{p}_{2}&=\left(-2x_{1}x_{2}|x|^{-4}\right)p_{1}+\left(|x|^{-2}-2x_{2}^{2}|x|^{-4}\right)p_{2}, \\
		\tilde{q}_1&=\left(|x|^{-2}-2x_1^2|x|^{-4}\right)q_1+\left(-2x_1x_2|x|^{-4}\right)q_2, 
	\end{align*}
	and
	\[ \tilde{q}_{2}=\left(-2x_{1}x_{2}|x|^{-4}\right)q_{1}+\left(|x|^{-2}-2x_{2}^{2}|x|^{-4}\right)q_{2}.\]
	It's easy to see that 
	\[\tilde{p}_{1}^{2}+\tilde{p}_{2}^{2}+\tilde{q}_{1}^{2}+\tilde{q}_{2}^{2}=|x|^{-4}\left(p_{1}^{2}+p_{2}^{2}+q_{1}^{2}+q_{2}^{2}\right),\]
	and
	\[\tilde{p}_1\tilde{q}_2-\tilde{p}_2\tilde{q}_1=-|x|^{-4}\left(p_1q_2-p_2q_1\right).\]
	
	Since $w=(p,q)$ is extended exterior $(K,K^{\prime},\alpha_0)$  quasiconformal over $\mathbb{R}^2\setminus\overline{B}_1(0)$, we deduce by Definition \ref{eeq} that $p$ and $q$ satisfy \eqref{ceeq} in $\mathbb{R}^2\setminus \overline{B}_1(0)$ for some $K\geq1$. Therefore, we obtain that in $B_1(0)\setminus\{0\}$,
	\begin{equation}\label{kqc}
		\begin{aligned}
			\tilde{p}_1^2+\tilde{p}_2^2+\tilde{q}_1^2+\tilde{q}_2^2&\leq |x|^{-4}\left(2K(p_1q_2-p_2q_1)+K^{\prime}|x|^{2\alpha_0}\right)\\
			&=2K(\tilde{p}_2\tilde{q}_1-\tilde{p}_1\tilde{q}_2)+K^{\prime}|x|^{-2\left(2-\alpha_0\right)}
		\end{aligned}
	\end{equation}
	This completes the proof.
\end{proof}

\noindent\textbf{Proof Theorem \ref{abw}.} Without loss of generality, suppose  that $B_1(0) \subset \Omega$. Let $\tilde{p}$ and $\tilde{q}$ be the Kelvin transform of $p$ and $q$ respectively . Let $\hat{\Omega} = \left\{ x/|x|^2 \middle| x \in \mathbb{R}^2 \setminus \overline{\Omega} \right\}$ and for any $\Omega^{\prime} \supset \Omega$, $\tilde{\Omega} = \left\{ x/|x|^2 \middle| x \in \mathbb{R}^2 \setminus \overline{\Omega^{\prime}} \right\}$. Then by Lemma \ref{klem}, $\tilde{w} = (\tilde{q}, \tilde{p})$ is $(K,K^{\prime})$ quasiconformal in $\hat{\Omega} \setminus \{0\}$ with $K > 1$. Since $|w| \leq M$ implies $|\tilde{w}| \leq M$. Since $\alpha_0>1$, then $\lambda_0<1$, applying Theorem \ref{fsthm} and Remark \ref{wbd} to $\tilde{w}$, we know that
 $\tilde{w}(x)$ has a limit $\tilde{w}(0)$ at $0$, then for all $x \in \tilde{\Omega}$,
$$
|\tilde{w}(x) - \tilde{w}(0)| \leq C|x|^\alpha,  $$
with  $\alpha=\min \left\{K-(K^2-1)^{\frac12}, \alpha_0-1\right\}$ and $C$ depends only on $K$, $K^{\prime}$, $d$ and $M$.  Transforming back to exterior domain, we have 
$$
|w(x) - w(\infty)| \leq C|x|^{-\alpha}, \, x \in \mathbb{R}^2 \setminus \overline{\Omega'}.$$
The theorem is therefore proved. \qed


Considering the linear elliptic equation
\begin{equation}\label{leq}
	L(u)=a_{11}(x)u_{11}(x)+2a_{12}(x)u_{12}(x)+a_{22}(x)u_{22}(x)=f(x),
\end{equation}
where $L$ is uniformly elliptic, that is, there exist constants $0<\lambda\leq\Lambda$ such that
$$\lambda(\xi_1^2+\xi_2^2)\leq a_{11}\xi_1^2+2a_{12}\xi_1\xi_2+a_{22}\xi_2^2\leq\Lambda(\xi_1^2+\xi_2^2),\quad \forall\xi=(\xi_1,\xi_2)\in\mathbb{R}^2.$$

For uniformly elliptic equation \eqref{leq} in a bounded domain $\Omega$ of $\mathbb{R}^2$, it follows from the interior H\"{o}lder estimate of $(K,K^{\prime})$  quasiconformal mappings that its bounded solutions with bounded $f$ have interior $C^{1,\alpha}$ estimate \cite[Theorem 12.4]{GT}.

For uniformly elliptic equation \eqref{leq} over exterior domain in $\mathbb{R}^2$, we can establish the gradient H\"{o}lder estimate and the gradient asymptotic behavior of solutions at infnity by the virtue of Theorem \ref{abw}.

\begin{thm}\label{cadu}
	Let $\Omega$ be a bounded domain of $\mathbb{R}^2$ and  $f(x) \leq C|x|^{-\alpha_0}$ for some $\alpha_0>1$, $ x\in \mathbb{R}^2\setminus \overline{\Omega}$. Suppose $v \in C^2(\mathbb{R}^2\setminus\overline{\Omega})$ is a solution of equation \eqref{leq} in 
	$\mathbb{R}^2\setminus\overline{\Omega}$ and $|Dv|\leq M$. Then for any $\Omega^{\prime}\supset\supset\Omega $ with $d=$dist$(\Omega,\partial\Omega^{\prime})$,
 $Dv(x)$ has a limit $Dv(\infty)$ at infinity with
	\begin{equation}\label{abdu}
		|Dv(x)-Dv(\infty)|\leq C|x|^{-\alpha},\quad x\in\mathbb{R}^2\setminus\overline{\Omega^\prime},
	\end{equation}
	where $\alpha$ depends only on $\lambda$, $\Lambda$ and $\alpha_0$, $C$ depends only on $\lambda$, $\Lambda$, $d$, $\alpha_0$ and $M$.
\end{thm}
\begin{rem}
	The results in Theorem \ref{cadu} are also valid for $v\in W^{2,2}(\mathbb{R}^2\setminus\overline{\Omega}).$
\end{rem}

\noindent\textbf{Proof of Theorem \ref{cadu}.} Assume without loss of generality that $\lambda=1$, Let $p=v_1,q=v_2.$ By
equations \eqref{leq}, we have 
$$p_1^2+p_2^2\leq a_{11}p_1^2+2a_{12}p_1p_2+a_{22}p_2^2=a_{22}J+fp_1,\quad  x\in\mathbb{R}^2\setminus\overline{\Omega}$$
where $J=p_2q_1-p_1q_2$ and
$$q_1^2+q_2^2\leq a_{11}J+fq_2,\quad x\in\mathbb{R}^2\setminus\overline{\Omega}.$$
Noticing that $2\leq a_{11}+a_{22}\leq2\Lambda$, we get
\begin{align*}
	p_1^2+p_2^2+q_1^2+q_2^2&\leq(a_{11}+a_{22})J+f(p_1+q_2)\\
	&\leq2\Lambda J+f(p_1+q_2)\\
	&\leq 2\Lambda J+\Lambda C|x|^{-\alpha_0}(|p_1|+|q_2|)\\
	&\leq2\Lambda J+\frac{1}{2}\left(|p_1|^2+|q_2|^2\right)+4C^2\Lambda^2|x|^{-2\alpha_0}, \
	x\in\mathbb{R}^2\setminus\overline{\Omega},
\end{align*}
 then we get
\[p_1^2+p_2^2+q_1^2+q_2^2\leq 4\Lambda J+8C^2\Lambda^2|x|^{-2\alpha_0},\]
which implies that $w=(q,p)$ is extended exterior $(K,K^{\prime},\alpha_0)$ quasiconformal over $\mathbb{R}^2\setminus\overline{\Omega}$ with $K=2\Lambda$ and $K^{\prime}=8C^2\Lambda^2$.

Since $|Dv|\leq M$ in $\mathbb{R}^2\setminus\overline{\Omega}$, Theorem \ref{abw}  asserts that for any $\Omega^{\prime}\supset\supset\Omega$, there exist $\alpha$ and $C$ such that
 $Dv(x)$ tends to a limit $Dv(\infty)=(p(\infty),q(\infty))$ at infınity with
$$|Dv(x)-Dv(\infty)|\leq C|x|^{-\alpha}, \ x\in\mathbb{R}^2\setminus\overline{\Omega^{\prime}}.$$
The theorem is therefore proved. \qed

Based on the above preparation work, we can find the limit $A$ of the Hessian $D^2u$ of \eqref{eq} at infinity and estimate the decay rate of $|D^2u-A|$ at infinity.

\begin{thm}\label{wdecay}
  Let $u$ be a smooth solution of 
  \begin{equation}\label{thm2.10eq}
  	G(D^2u)=f(x),\quad  x\in \mathbb{R}^2\setminus\overline{\Omega},
  \end{equation}
 where  $G\in C^2$  is a fully nonlinear uniformly elliptic operator with ellipticity constants $\lambda$ and $\Lambda$,    $\Omega$ is a bounded domain in $\mathbb{R}^2$.
   Suppose $\|D^2u\|_{L^{\infty}(\mathbb{R}^2\setminus\overline{\Omega})}\leq M$ and $f$ satisfies \eqref{abf} with $\beta>0$, then  there exists a  matrix $ A$ such that $G(A)=0$ and
	\[D^2u=A+O(|x|^{-\varepsilon}) \ \text{ as } |x| \rightarrow \infty,\]
	which implies
	\begin{equation}
		\left|u(x) - \frac{1}{2}x^\mathsf{T} Ax\right| \leq C|x|^{2-\varepsilon}, \ |x|\geq R_0,
	\end{equation}
	where $\varepsilon \in (0, 1)$ and  $C$ are positive constants depending only on $\lambda$, $\Lambda$, $\beta$ and $M$, $R_0$ is a large constant.
\end{thm}
\begin{proof}
	Taking derivative with respect to $x_k$, $k=1,2$ for
	on both sides of  \eqref{thm2.10eq} to obtain
	\begin{equation}\label{dleq}
		a_{ij}(x)v_{ij}(x)=f_k,\ x\in\mathbb{R}^2\setminus\overline {\Omega},
	\end{equation}
	where $a_{ij}(x)=G_{M_{ij}}\left(D^{2}u(x)\right)$ and $v(x)=u_k(x)$.
	
Since $\|D^2u\|_{L^{\infty}(\mathbb{R}^2\setminus\overline{\Omega})}\leq M$, then $|Dv|<M$. We may assume  $|f_k|\leq C|x|^{-1-\beta}$ in $\mathbb{R}^2\setminus\overline {\Omega}$. Applying Theorem \ref{cadu} to \eqref{dleq}
	in $\mathbb{R}^2\setminus\overline{\Omega}$. 
	Therefore,
	we have that $Dv(x)$ tends to a limit $Dv(\infty)$ at infınity and fix a domain $\Omega^\prime\supset\supset\Omega$, then
	$$|Dv(x)-Dv(\infty)|\leq C|x|^{-\varepsilon},\ x\in\mathbb{R}^2\setminus\overline{\Omega'}.$$
	Then by the arbitrarity of $k$, we conclude that there exists a  matrix $A$  such
	that $D^2u(x)\to A$ as $|x|\to\infty$ and
	$$D^2u=A+O(|x|^{-\varepsilon}), \ \text{as} \ |x|\to\infty.$$
	
	It follows that
	$$\left|u(x)-\frac12x^\mathsf{T} Ax\right|\leq C|x|^{2-\varepsilon}\  \ |x|\geq R_0,$$
	where $\varepsilon\in(0,1)$ and $C>0$ depending only on $\lambda$, $\Lambda$, $\beta$ and $M$.
\end{proof}

\begin{rem}
	Note that in the proof of the aforementioned theorem, we only take the derivative of the equation once,  only need  the boundedness of $D^2u$, without invoking the concavity of the operator $G$. This is clearly different from \cite[Theorem 2.1]{LLY2020},
	 it's  an essential assumption that $G$ is either convex, or
	concave, or the level set is convex.
\end{rem}

In the following, we present a Theorem of  the limit of $D^2u$ that holds for  $n \geq  3$.
\begin{lem}\label{gth}
	Let $n \geq  3$.	Under the assumption of Theorem \ref{wdecay}  and suppose $\theta+f(x)\geq (n-2)\pi/2+\delta$.
If $f$ satisfies \eqref{abf} with $\beta>0$  sufficiently lagre. 
 Then there exists a  matrix $ A\in \mathcal{A}$ such that
	\[D^2u(x)\to A\quad \text{as}~~|x|\to\infty.\]
\end{lem}
\begin{proof}
	This proof mainly follows the idea of \cite[Theorem 4]{J2020}, with the key difference being the choice of the barrier function.
	
	Let $\omega(x) = u_{ee}(x)$, where $e \in \partial B_1$. The goal is to show that $\omega$ converges at infinity.

	Denote
	$$\overline{\omega}=\limsup_{|x|\to\infty}\omega(x),\quad\underline{\omega}=\liminf_{|x|\to\infty}\omega(x).$$
	It's enough to prove that $\overline{\omega}=\underline{\omega}.$
	
	Now we argue this by contradiction. If it is wrong, we have $\overline{\omega}-\underline{\omega}=:5d>0.$ Clearly, for any
	$0<\varepsilon<d$, there exists some large constant $R=R(\varepsilon)>$l such that
	$$\underline{\omega}-\varepsilon\leq \omega(x)\leq\overline{\omega}+\varepsilon $$
	$\operatorname{for}$ all $x\in B_{R/2}^{\complement}$, and also there exists a sequence of $\underline{x}_k$ in $B_{R/2}^{\complement}$, \( |\underline{x}_k| \leq |\underline{x}_{k+1}| \) and tending to $\infty$ , such that
	$$\omega\left(\underline{x}_k\right)\leq\underline{\omega}+\varepsilon, $$
	for all $k\in\mathbb{Z}_+.$ 
	\begin{clm}
		There exists a point $\overline{x} $ on the sphere $\partial B_{|\underline{x}|}$ for at least one $\underline{x}\in\{\underline{x}_k\}$, such that
		$$\omega\left(\overline{x}\right)\geq\overline{\omega}-\varepsilon.$$
	\end{clm}
	\begin{proof}
		Otherwise, $\omega<\overline{\omega}-\varepsilon$ on the spheres $\partial B_{|\underline{x}_k|}$ for all $k\in\mathbb{Z}_+.$
		
By the concavity of  $F$,  for all $ k$ large,
		\[L\omega\geq f_{ee},\quad  x\in B_{|\underline{x}_k|}^{\complement},\]
where $L\phi=F_{ij}\phi_{ij}$.
	
		Consider the following  supersolution:
	\[H(x)=C_{\star}\left(\hat{R}^{-\beta}-|x|^{-\beta}\right),\quad |x|\geq \hat{R}\]
	where $C_{\star}$  and $\hat{R}\geq R$ are constants to be determined. A direct computation yields
	\[H_{ij}(x)=-C_{\star}(\beta+2)\beta|x|^{-\beta-4}x_ix_j+C_{\star}\beta|x|^{-\beta-2}\delta_{ij},\]
	and
	\begin{align*}
		LH&=C_{\star}\beta|x|^{-\beta-4}  \left[-(\beta+2)F_{ij}x_ix_j+|x|^2\sum_{i=1}^n F_{ii}\right]\\
		&\leq C_{\star}\beta|x|^{-\beta-2}  \left[-\frac{\beta+2}{1+M^2} +n\right]\\
		&\leq -C(n,\theta,\beta,M,C_{\star}) |x|^{-\beta-2}.
	\end{align*}
	The last second inequality holds since for sufficiently large \( \beta \) such that \( \beta\geq \left(1+M^2\right)(n+\delta_0)-2 \) for some constant \( \delta_0 \). Therefore, by choosing \( C_{\star} \) sufficiently large, we can have \( LB(x) \leq -|f_{ee}(x)| \) for sufficiently large \( \tilde{R} \geq \hat{R} \) and \( |x| \geq \tilde{R} \).
	
	Furthermore, after fixing \( C_{\star} \), we may choose \( \hat{R} \) sufficiently large such that \( C_{\star}\hat{R}^{-\beta} < \frac{\varepsilon}{2} \). Indeed, by assumption, \( \omega < \overline{\omega} - \varepsilon \) on each \( \partial B_{|\underline{x}_k|}, k \in \mathbb{Z}_{+} \). Clearly, for all $k\geq k_0$ such that $|\underline{x}_{k}|\geq \hat{R}$, we have 
	\[
	\omega < \overline{\omega} - \varepsilon + H \quad \text{on} \ \partial B_{|\underline{x}_{k}|}.
	\]
	
	By the comparison principle, for all \( k \geq k_1\geq k_0 \) such that \( |\underline{x}_{k_1}| \geq \tilde{R}\geq \hat{R} \geq \left( 2C_{\star}/\varepsilon \right)^{\frac{1}{\beta}} \), we have
	\[
	\omega<\overline{\omega} - \epsilon + C_{\star}\hat{R}^{-\beta} \leq \overline{\omega} - \frac{\epsilon}{2} \quad \text{in } B_{|\underline{x}_{k+1}|} \setminus B_{|\underline{x}_{k}|},
	\]
	i.e., \( \omega(x) \leq \overline{\omega} - \frac{\epsilon}{2} \) for \( x \in B_{|\underline{x}_{k_1}|}^{\complement} \), which leads to a contradiction.
		\end{proof}
			Applying  the Evans--Krylov estimate to $\omega=u_{ee}$ in $B_{|\underline{x}|/2}(\underline{x})$ (cf. \cite{GT}) with the decay condition \eqref{abf}, we have
		\begin{align*}
			\mathrm{osc}_{B_{\gamma|\underline{x}|(\underline{x})}}   u_{ee} & \leq C\left(\frac{2\gamma\left|\underline{x}\right|}{\left|\underline{x}\right|}\right)^\alpha \left\{\mathrm{osc}_{B_{\gamma|\underline{x}|(\underline{x})}}D^2u+\left|\underline{x}\right|\left|Df\right|_{0,B_{\left|\underline{x}\right|/2}\left(\underline{x}\right)}+\left|\underline{x}\right|^2\left|D^2f\right|_{0,B_{\left|\underline{x}\right|/2}\left(\underline{x}\right)}\right\} \\
			& \leq2C\gamma^\alpha(K+1) \\
			& \leq d,
		\end{align*}
		where	$\alpha=\alpha(n,\theta),\gamma=\gamma(n,\theta,K,d)=:\min\{1/10,(d/(2C(K+1)))^{1/\alpha}\}$. Thus, \[\omega(x)\leq\underline{\omega}+\varepsilon+d\leq\overline{\omega}-3d\quad\mathrm{for~}x\in B_{\gamma|\underline{x}|}\left(\underline{x}\right),\]
		which yields
		$$\overline{\omega} - \omega(x) \geq 3d \quad \text{for } x \in B_{\gamma|\underline{x}|}(\underline{x}).$$
		Let $v(x) = \overline{\omega} + \varepsilon - \omega(x)$. Then
		$$
		L v \leq -f_{ee} \quad \text{in } |x| \geq R.$$
		Applying the weak Harnack inequality to $v$ in $B_{(1+3\gamma)|\underline{x}|} \setminus \overline{B}_{(1-3\gamma)|\underline{x}|}$ (cf. \cite{GT}), we obtain that
		\begin{align*}
			\left(\frac{1}{\left|B_{\gamma\left|\underline{x}\right|}\right|}\int_{B_{(1+\gamma)\left|\underline{x}\right|}\setminus\overline{B}_{(1-\gamma)\left|\underline{x}\right|}}v^\delta\right)^{1/\delta} & \leq C\left\{\inf_{B_{(1+\gamma)\left|\underline{x}\right|}\setminus\overline{B}_{(1-\gamma)\left|\underline{x}\right|}}v+\left|\underline{x}\right|\left\|f_{ee}\right\|_{L^n\left(B_{(1+3\gamma)\left|\underline{x}\right|}\setminus\overline{B}_{(1-3\gamma)\left|\underline{x}\right|}\right)}\right\} \\
			& \leq C\left\{2\varepsilon+\left|\underline{x}\right|^{2}\sup_{B_{(1+3\gamma)\left|\underline{x}\right|}\setminus\overline{B}_{(1-3\gamma)|\underline{x}|}}f_{ee}\right\} \\
			& \leq C\left(2\varepsilon+\left|\underline{x}\right|^{-\beta}\right) \\
			& \leq3C\varepsilon
		\end{align*}
		for $|\underline{x}|$ large, where $C$ is independent of $\varepsilon$, as $\beta>0$. Then $3d\leq 3C\varepsilon$. Letting $\varepsilon\rightarrow 0$, we get $d=0$, a contradiction.
	\end{proof}

\begin{rem}\label{betal}
	 We point out that $\beta$ in Lemma \ref{gth} is always graet than $n$ as $n\geq 4$. Indeed, $D^2u\to A$ as $|x|\to \infty$ and by \eqref{eq}, $\lambda_{\max}(A)>1$, then we have $M>1$ and  the inequality
	 \[\beta>\left(1+M^2\right)(n+\delta_0)-2>n\left(1+M^2\right)-2\geq 2n-2>  n\]
	 holds when $n\geq 4$. 
\end{rem}	

\subsection{Some useful lemmas}\label{sec2.4}
In this subsection, we collect some  preliminary lemmas which
will be mainly used in Section \ref{sec4}.

To study the  boundary behavior of the exterior Dirichlet problem, we need the following Lemma.
\begin{lem}\label{bdsublem}
	Let \( \Omega \) be a bounded, strictly convex domain in \( \mathbb{R}^n \), \( n \geq 3 \), \( \partial \Omega \in C^{1,1} \). Let \( \varphi \) be semi-convex with respect to $\partial \Omega$. Assume \( K > 0 \) and let \( A \) be a positive definite and symmetric matrix. There exists some constant \( C \), depending only on \( n \), \( \varphi \), \( K \), the upper and lower
	bound of \( A \),  and the \( C^{1,1} \) norm of \( \partial \Omega \), such that for every \( \xi \in \partial \Omega \), there exists \( \overline{x}(\xi) \in \mathbb{R}^n \) satisfying \( |\overline{x}(\xi)| \leq C \) and  
	\[
	Q_{\xi} < \varphi \quad \text{on } \partial \Omega \setminus \{\xi\},
	\]
	where  
	\[
	Q_{\xi}(x) = \varphi(\xi) + \frac{K}{2}\left[(x - \overline{x}(\xi))^T A(x - \overline{x}(\xi)) - (\xi - \overline{x}(\xi))^\mathsf{T} A(\xi - \overline{x}(\xi))\right]
	\]
	for \( x \in \mathbb{R}^n \).  Moreover, there exists $\overline{c}=\overline{c}(\Omega,K,A,\varphi)$ such that
	\begin{equation}\label{bdsubbd}
		Q_{\xi}(x) \leq \frac{K}{2}x^{\mathsf{T}}Ax+\overline{c}
	\end{equation}
	for every $\xi \in \partial \Omega$ and $x\in \mathbb{R}^n\setminus \Omega$.
\end{lem}
\begin{proof}
The proof can be seen a simple extension by \cite[Lemma 3.6]{BW2024}.  We only to  choose $\overline{x}(\xi)$ satisfying $\overline{x}(\xi)A= \left(-p(\xi)/K, R(\xi)\right)$, where  \(p = p(\xi)\) and \(R(\xi)\) are bounded on $\partial \Omega$ provided $\partial \Omega\in C^{1,1}$ (see \cite[Lemma 2.5]{WB2024arxiv}).  Consequently,  $\overline{x}(\xi)$ is bounded and  \( Q_{\xi} < \varphi \) on \(\partial\Omega \setminus \{\xi\}\) for every $\xi\in \partial \Omega$. \eqref{bdsubbd} is clear, since $\xi$ and $\overline{x}(\xi)$ are bounded.
\end{proof}

Now we introduce the following well known Perron’s method, appread in \cite{CIL1992,I1989,I1989,BLL2014} and \cite{LZ2019}. 

\begin{lem}[Perron's method]
	\label{pemd}
 Assume \( S \in C^1(\mathbb{R}^n) \) and \( S_{\lambda_i}(\lambda) > 0 \), \(\forall \lambda=(\lambda_1,\cdots,\lambda_{n})\in \mathbb{R}^n\). Let \(D \subset \mathbb{R}^n\) be a bounded, strictly convex  domain, \( \partial D \in C^{1,1}\), \(\varphi \in C^0(\partial D)\) and let \(\underline{u}, \overline{u} \in C^0(\mathbb{R}^n\setminus D)\) satisfying  
	\[
	S \left( \lambda (D^2 \underline{u}) \right) \geq 0 \geq S \left( \lambda (D^2 \overline{u}) \right)
	\]
	in \(\mathbb{R}^n\setminus \overline{D}\) in the viscosity sense. Suppose \(\underline{u} \leq \overline{u}\) in \(\mathbb{R}^n\setminus D\), \(\underline{u} = \varphi\) on \(\partial D\) and additionally  
	\[
	\lim_{|x| \to +\infty} (\underline{u} - \overline{u})(x) = 0.
	\]
	 Then  
	\begin{align*}
		u(x) &:= \sup \{ v(x) \;|\; v \in C^0(\mathbb{R}^n\setminus D), \; \underline{u} \leq v \leq \overline{u} \text{ in } \mathbb{R}^n\setminus D, \; S \left( \lambda (D^2 v) \right) \geq 0 \text{ in } \mathbb{R}^n\setminus \overline{D}\\
		&\quad  \text{ in the viscosity sense}, \; v = \varphi \text{ on } \partial D \}
	\end{align*}
	is the unique viscosity solution of the Dirichlet problem  
	\[
	\begin{cases}
		S \left( \lambda (D^2 u) \right) = 0 & \text{in } \mathbb{R}^n\setminus \overline{D}, \\
		u = \varphi & \text{on } \partial D.
	\end{cases}
	\]
\end{lem}
	
The solvability of  Dirichlet problem with continuous boundary data for the Lagrangian mean curvature equation as follows.
\begin{thm}[{\cite[Theorems 1.1 and 1.2]{AB2024}}]\label{lsdp}
	Suppose that \(\varphi \in C^0(\partial D)\) and \(\psi : D \rightarrow \left[(n-2)\pi/2 + \delta,n\pi/2\right)\) is in \(C^{1,1}(D)\), where \(D\)
	is a uniformly convex, bounded domain in \(\mathbb{R}^n\) and \(\delta > 0\). Then there exists a unique solution \(u \in
	C^{2,\alpha}(D) \cap C^0(\overline{D})\) to the Dirichlet problem 
	\[
	\begin{cases}
	 \sum_{i=1}^n \arctan \lambda_i(D^2 u) = \psi(x) & \text{in } D, \\
		u = \varphi & \text{on } \partial D.
	\end{cases}
	\]
	
	If \(\psi \in \left(-n\pi/2, n\pi/2\right)\) is a constant and the other hypotheses remain unchanged, then the above problem admits a unique viscosity solution \(u \in C^0(\overline{D})\).
\end{thm}

\section{Proof of Theorems \ref{D2ue} and \ref{ng3r} }\label{sec3}
In this section, we prove Theorem \ref{D2ue} and \ref{ng3r} by standard arguments developed by \cite{CL2003,BLZ2015,J2020}.

 For $n=3$,  by Jia's work for fully
nonlinear equations over exterior domains and Remark 3 in \cite{BLW2024}, we already have
\begin{thm}[{\cite[Theorem 1]{J2020}}]\label{Jthm1}
	Let $u$ be a smooth solution of fully nonlinear equation
	$$F\left(D^2u\right)=f(x)\quad \text{in}\ \mathbb{R}^n\setminus\overline{B}_1,$$
	where	$n\geq3, F\in C^m(\mathbb{R}^{n})$ is concave and uniformly elliptic, and $f\in C^m(\mathbb{R}^n\setminus\overline{B}_1)$ satisfies
	$$f(x)=O_m\left(|x|^{-\beta}\right)\quad  \text{as}\ |x|\to\infty,$$
	where	$m\geq2$ and $\beta>2$. Suppose that
	$$D^2u\to A\quad \text{as} \ |x|\to\infty,$$
	where $A$ is some symmetric positive definite matrix with $F(A)=0$. Then there exists a unique
	quadratic polynomial 
	$$Q(x)=\frac12x^\mathsf{T} Ax+b^\mathsf{T} x+c$$
	such that
	\begin{equation}\label{Jthm}
		u-Q=\begin{cases}O_{m+1}\left(|x|^{2-\min\{\beta,n\}}\right),&\text{if} \ \beta\neq n,\\
				O_{m+1}\left(|x|^{2-n}\ln |x|\right), &\text{if}\ \beta= n,
		\end{cases}
	\end{equation}
as $|x|\to \infty$,	where $b\in \mathbb{R} ^n$ is some vector $ c\in \mathbb{R}$ is some constant.
\end{thm}

By virtue of this theorem, we can finish the proof for $n\geq 3$.

\noindent\textbf{Proof of Theorem \ref{ng3r}:} By Lewy Rotation in Section \ref{sec2.2}, Eq. \eqref{npeq} is uniformly second elliptic equation and $F$ is concave. Then by Lemmas \ref{abf11} and \ref{gth}, there exists some symmetric positive definite matrix $\tilde{A}$ with $F(\tilde{A})=\tilde{\theta}$, such that $D_{\tilde{x}}^2\tilde{u}(\tilde{x})\to \tilde{A}$ as $|\tilde{x}|\to +\infty$.    Consequently,  Theorem \ref{Jthm1} (for $m=2$) implies there exists $\tilde{b}\in \mathbb{R}^n$ such that $D_{\tilde{x}}\tilde{u}(\tilde{x})$
satisfies the condition in   Proposition \ref{rtp} (v). Therefore, $D^2u$ is boundede in $\mathbb{R}^n\setminus \overline{\Omega}$ and there exixsts $A\in \mathcal{A}$ such that $D^2u(x)\to A$  as $|x|\to +\infty$. Using Theorem  \ref{Jthm1} again and by Remark \ref{betal}, Theorem \ref{ng3r} is proved.   \qed

For $n=2$,  based on Proposition \ref{rtp} v and  Theorem \ref{wdecay}, the asymptotic behavior of $u$ can be found by standard
arguments (cf. \cite{CL2003,BLZ2015}), while for readers’ convenience, we show the detailed proofs in the following.

In the following, we gradually give higher order asymptotic behavior for supercritical phase Lagrangian mean
curvature equations.

\begin{lem}\label{ab1}
	  Let $u$ be a smooth solution of  equation \eqref{eq}
	with $\theta+f(x)> 0$.
	   Suppose there exists $A\in \mathcal{A}$ such that
	\begin{equation}\label{D2uee}
		\left|u(x) - \frac{1}{2}x^\mathsf{T}Ax\right| \leq C_1|x|^{2-\varepsilon} \ \quad |x|\geq R_0,
	\end{equation}
for some constants $\varepsilon$, $C_1$ and large $R_0$. If	$f$ as in \eqref{abf} for some constant $\beta> 0$.
	
	Let
	$$w(x) := u(x) - \frac{1}{2}x^{\mathsf{T}}Ax,$$
	then there exist $C(\theta, R_0, \varepsilon, C_1,A, \beta) > 0$ and $R_1(\theta,R_0, \varepsilon,A, C_1,\beta) > R_0$ such that for any $\alpha \in (0, 1)$
	\begin{equation}\label{Dkuee}
		\begin{cases}
			|D^k w(x)| \leq C|x|^{2-k-\varepsilon_\beta}, & k = 0, \cdots, m+1, \quad |x| > R_1 \\
			\frac{|D^{m+1}w(x_1) - D^{m+1}w(x_2)|}{|x_1 -x_2|^\alpha} \leq C|x_1|^{1-m-\varepsilon_\beta-\alpha}, & |x_1| > R_1, \quad x_2 \in B_{|x_1|/2}(x_1)
		\end{cases}
	\end{equation}
	where $\varepsilon_\beta = \min\{\varepsilon, \beta\}$.
\end{lem}
\begin{proof}
For sufficiently large \( R := |x| > 1 \), set
\[
u_R(y) := \left( \frac{4}{R} \right)^2 u \left( x + \frac{R}{4} y \right)
\]
and
\[
w_R(y) := \left( \frac{4}{R} \right)^2 w \left( x + \frac{R}{4} y \right)
\]
in \( B_2 \).

In view of  \eqref{D2uee},   it's clear that 
\[
\max_{y \in B_2} |u_R(y)| \leq \frac{16}{R^2} \max_{z \in B_{3R} \setminus B_{R/2}} |u(z)| \leq C
\]
for some \( C > 0 \) independent of \( R \) and
\[
\|w_R\|_{L^\infty(B_2)} \leq \frac{16}{R^2} \left\| u - \frac{1}{2}x^TAx \right\|_{L^\infty(B_{3R} \setminus B_{R/2})} = O(R^{-\varepsilon}).
\]

Then,
\[
F(D^2 u_R(y)) = F \left( D^2 u \left( x + \frac{R}{4} y \right) \right) =\theta+ f \left( x + \frac{R}{4} y \right) =: f_R(y).
\]

 By a direct computation and condition \eqref{abf},
\[
\| f_R - \theta \|_{C^m(B_2)} \leq CR^{-\beta}.
\]
for some positive constant $C$ independent of $R$. Hence, by the gradient estimate and Hessian estimate as in \cite{BMS2022}, there exists $C$ independent of $R$ such that 
\[\|u_{R}\|_{C^2(B_{1.6})}\leq C\quad \text{and}\quad \|w_{R}\|_{C^2(B_{1.6})}\leq C.\]
Consequently, $F$ is uniformly elliptic with respect to all $u_R$ and concave. By the Evans--Krylov estimate and the Schauder theory, for any $0 < \alpha < 1$, we have
\[\|u_{R}\|_{C^{2,\alpha}(B_{1.5})}\leq C.\]

	The difference between \eqref{eqR} and $F(A) = \theta$ gives,
	\[\tilde{a}_{ij}^R \partial_{ij} w_R = f_{R}(y) - \theta=O\left(R^{-\beta}\right)\]
	where $\tilde{a}_{ij}^R(y) = \int_0^1 F_{M_{ij}}\left(tD^2w_R(y) + A\right)dt$. 
	
	Since $\tilde{a}_{ij}^R$,  and $ f_{R}$ are bounded in $C^\alpha$ norm and $\frac{1 }{C}\leq \tilde{a}_{ij}^R\leq C$,   Schauder  estimate gives
	\begin{equation*}
		\|w_R\|_{C^{2,\alpha}(B_{1})}\leq C\left(\|w_R\|_{L^\infty(B_{1.2})}+\|f_{R}-\theta\|_{C^\alpha(B_{1.2})}\right)\leq CR^{-\varepsilon_\beta}.
	\end{equation*}
	By differentiating on the above equation  and Schauder estimate, we have
	\[\|w_R\|_{C^{m+1,\alpha}\left(B_{\frac{1}{2^{m-1}}}\right)}\leq CR^{-\varepsilon_\beta}.\]
	Thus,  a direct calcaulation gives \eqref{Dkuee}.
\end{proof}

The following are two basic lemmas. Please see \cite{BLZ2015,LL24} for a short proof.
\begin{lem}\label{hmc}
	Suppose $f(x)=O(|x|^{-\hat{\beta}})$ as $|x|\to\infty$ with $\hat{\beta}>1$. Then for any $\epsilon>0$, the equation
	$$\Delta u(x)=f(x) \text{ in } \mathbb{R}^2\setminus\overline{B}_1(0)$$
	has a solution $u(x)=O_2(|x|^{2-\hat{\beta}+\epsilon})$ as $|x|\to\infty$.
\end{lem}

\begin{lem}\label{hm2}
	Let $u$ be a smooth solution of
	$$\Delta u(x) = 0,\ x \in \mathbb{R}^2 \setminus \overline{B}_1(0)$$
and $u(x) = O(|x|^{\xi})$ as $|x|\to \infty$	for some $-1 < \xi < 2$. Then
	\begin{equation}\label{uab}
		u(x) = b^\mathsf{T} x + d \ln |x| + c + O(|x|^{-1}) \text{ as } |x| \to \infty,
	\end{equation}
	where $b \in \mathbb{R}^2$, $c, d \in \mathbb{R}$. Moreover, we have $b = 0 $ provided $0 < \xi < 1$, $b =d= 0 $  provided $\xi =0$ and $b =c=d= 0 $  provided $-1<\xi < 0$.
\end{lem}

Next we prove a lemma that improves the estimates in Lemma \ref{ab1}.
\begin{lem}\label{ab2}
	Under the same assumptions of Lemma \ref{ab1} and let $R_1$ be the large constant
	determined in the proof of Lemma \ref{ab1}. If in addition $2\varepsilon<1$ and $\beta >1$.
	Then for  any $\overline{\varepsilon } < 2\varepsilon < 1$, 
	\begin{equation}
		\begin{cases}
			|D^kw(x)|\leq C|x|^{2-\overline{\varepsilon}-k},&|x|>2R_1,\quad k=0,\cdots,m+1,\\
			\frac{|D^{m+1}w(x_1)-D^{m+1}w(x_2)|}{|x_1-x_2|^\alpha}\leq C|x_1|^{1-m-\overline{\varepsilon}-\alpha},&|x_1|>2R_1,x_2\in B_{|x_1|/2}(x_1).
		\end{cases}
	\end{equation}
\end{lem}
\begin{proof}
	Applying $\partial_{k}$ to \eqref{eq} we have
	\begin{equation}\label{D1eq}
		a_{ij}\partial_{ij}(\partial_{k}u) = \partial_{k}f
	\end{equation}
	where $a_{ij} = F_{M_{ij}}\left(D^2u\right)$.
	
	Since it follows from Theorem \ref{wdecay} that $D^2u\to A$
	as $|x|\to \infty$, we know
	\[a_{ij}(x)\to F_{M_{ij}}(A)\quad \text{as} \ |x|\to \infty.\]
	Assuming without loss of generality that $F_{M_{ij}}(A)=\delta_{ij}$. Then, Lemma \ref{ab1} gives
	$$
	|a_{ij}(x) - \delta_{ij}| \leq \frac{C}{|x|^{\varepsilon}}, \quad |Da_{ij}(x)| \leq \frac{C}{|x|^{1+\varepsilon}}, \quad |x| > R_{1}$$
	and for any $\alpha \in (0,1)$
	$$\frac{|Da_{ij}(x_{1}) - Da_{ij}(x_{2})|}{|x_{1} - x_{2}|^{\alpha}} \leq C|x_{1}|^{-1-\varepsilon-\alpha}, \quad |x_{1}| > 2R_{1}, \quad x_{2} \in B_{|x_{1}|/2}(x_{1}).$$
	Then applying $\partial_{l}$ to \eqref{D1eq} and letting $h_{1} = \partial_{kl}u$ we further obtain
	$$
	a_{ij}\partial_{ij}h_{1} = \partial_{kl}f - \partial_{l}a_{ij}\partial_{ijk}u.$$
	Rewrite the equation above as
	\begin{equation}\label{lpf1}
		\Delta h_{1} = f_{1} := \partial_{kl}f - \partial_{l}a_{ij}\partial_{ijk}u - (a_{ij} - \delta_{ij})\partial_{ij}h_{1}.
	\end{equation}
	For any $\alpha\in(0,1)$, in view of \eqref{abf} and Lemma \ref{ab1}, we get
	\begin{equation}\label{abf1}
		\begin{cases}
			|f_1(x)|\leq C|x|^{-2-2\varepsilon} & |x|\geq 2R_1, \\
			\frac{|f_1(x_1)-f_1(x_2)|}{|x_1-x_2|^\alpha}\leq\frac{C}{|x_1|^{2+2\varepsilon+\alpha}}, & x_2\in B_{|x_1|/2}(x_1),\ |x_1|\geq 2R_1.
		\end{cases}
	\end{equation}
	By Lemma \ref{hmc} and \eqref{abf1}, there exists $h_2$ such that $\Delta h_2=f_1, |x|\geq 2R_1$  and 
	\begin{equation}\label{abh2}
		\begin{cases}
			|D^jh_2(x)|\leq C|x|^{-\overline{\varepsilon}-j} & |x|\geq 2R_1, \\
			\frac{|D^2h_2(x_1)-D^2h_2(x_2)|}{|x_1-x_2|^\alpha}\leq\frac{C}{|x_1|^{2+\overline{\varepsilon}+\alpha}}, & x_2\in B_{|x_1|/2}(x_1),\ |x_1|\geq 2R_1.
		\end{cases}
	\end{equation}
	
	Now we have
	$$\Delta(h_1 - h_2) = 0, \quad \mathbb{R}^2 \setminus B_{2R_1}.$$
	Note that by Lemma \ref{ab1},  $h_1(x)\to a_k\delta_{kl}$ as $x\to\infty$ (Suppose that $A$ is Diagonal with eigenvalues $\left(a_1,\cdots,a_n\right)$).  Then, by \eqref{abh2} and Lemma \ref{hm2},
	\[|h_1(x)-a_k\delta_{kl}-h_2(x)|\leq C|x|^{-1},\quad|x|>2R_1.\]
	
	Hence, 
	\[|h_1(x)-a_k\delta_{kl}|\leq C\left(|x|^{-1}+|x|^{-\overline{\varepsilon}}\right)\leq C|x|^{-\overline{\varepsilon}}\quad|x|>2R_1\]
	and then 
	\[|D^kw(x)|\leq C|x|^{2-k-\overline{\varepsilon}},\quad|x|>2R_1,\ k=0,1,2.\]
	Finally we apply Lemma \ref{ab1} to obtain the estimates on higher derivatives. Lemma \ref{ab2} is established.
\end{proof}

Now we complete the proof of Theorem \ref{D2ue}.  
We divide the proof into six steps.

\textbf{Step 1.}  Consider the new potential \eqref{npeq}, which is uniformly second elliptic equation and $F$ is concave in level set sense. Then by Lemma \ref{abf11} and Theorem \ref{wdecay}, there exists some symmetric positive definite matrix $\tilde{A}$ with $F(\tilde{A})=\tilde{\theta}$, such that
\[D_{\tilde{x}}^2\tilde{u}(\tilde{x})=\tilde{A}+O(|\tilde{x}|^{-\varepsilon^{\prime}})  \quad \text{as}\ |\tilde{x}|\to +\infty. \]
where $\varepsilon^{\prime}\in (0,1)$ is a constant depending only on $\theta$, $\beta$ and $\delta$. 

Consequently,  Lemmas \ref{ab1} and  \ref{ab2} are both valid for $\tilde{u}$ with $m=3$.
In order to get the limit of $D^2u$, Propsition \ref{rtp} (v) implies we  need to figure out the asymptotic behavior of $D_{\tilde{x}}\tilde{u}(\tilde{x})$. 

In the following, for ease of description, we shall  use $u$ to denote $\tilde{u}$.

\textbf{Step 2.}  Determining the linear term.

 We can repeat  Lemma \ref{ab2} $n_0$ times such that $2^{n_0} \varepsilon<1$ and $2^{n_0+1}\varepsilon>1$ (Assume $\varepsilon$ smaller if necessary to make both inequalities hold), provided that $\beta> 2$. Let $\varepsilon_1 = 2^{n_0} \varepsilon$, clearly we have $1 < 2\varepsilon_1 < 2$. Then, we have
\begin{equation}\label{abwn}
	\begin{cases}
		|D^kw(x)| \leq C|x|^{2-\varepsilon_1-k}, & k=0,\cdots,m+1, \quad |x|>2^{2n_0}R_1 \\
		\frac{|D^{m+1}w(x_1)-D^{m+1}w(x_2)|}{|x_1-x_2|^\alpha} \leq C|x_1|^{1-m-\varepsilon_1-\alpha}, & |x_1|>2^{2n_0}R_1, \ x_2 \in B_{|x_1|/2}(x_1).
	\end{cases}
\end{equation}
Taking the difference between the equation for \eqref{eq} and $F(A)=\theta$ we have
$$\tilde{a}_{ij}\partial_{ij}w = f,$$
where  $\tilde{a}_{ij}(x)=\int_0^1 F_{M_{ij}}\left(tD^2w(x) + A\right)dt$.   We rewrite the equation above as
$$
\Delta w = f_3 := f - (\tilde{a}_{ij} - \delta_{ij})\partial_{ij}w.$$
By \eqref{abwn} and $\beta> 2>2\varepsilon_1$, we have
$$|f_3(x)| \leq C|x|^{-2\varepsilon_1}, \quad |x| > 2^{2n_0}R_1.$$
In view of Lemma \ref{hmc}, there exists $h_3$ such that
$\Delta h_3=f_3$, $|x| > 2^{2n_0}R_1$ and 
$$|h_3(x)| \leq C|x|^{\varepsilon_2}, \quad |x| > 2^{2n_0+1}R_1$$
for some $\varepsilon_2 \in (0, 1)$. Since $w - h_3$ is harmonic on $\mathbb{R}^2 \setminus B_{2^{2n_0+1}R_1}$ and $w - h_3 = O(|x|^{2 - \varepsilon_1})$ as $|x|\to +\infty$. By Lemma \ref{hm2}, there exist $b \in \mathbb{R}^2$ such that
\begin{equation}\label{abwh3}
	w(x) - h_3(x) = b^\mathsf{T} x + O\left(\ln |x|\right) \quad \text{as}\ |x|\to +\infty.
\end{equation}

\textbf{Step 3.}  Determining the logarithm term and constant term.

Let
\[w_1(x) = w(x) - b^\mathsf{T} x.\]
Then by \eqref{abwh3}, it holds $|w_1(x)| \leq C |x|^{\varepsilon_2}$.   Since $\beta> 2>2-\varepsilon_2$, Applying  Lemma \ref{ab1} again,  we have
\begin{equation}\label{Dw1}
	|D^k w_1(x)| \leq C |x|^{\varepsilon_2 - k}, \quad k = 0, \cdots, m + 1, \quad |x| > 2^{2n_0+1}R_1.
\end{equation}

Therefore, for $k=1$, we prove the  desired asymptotic behavior of $D_{\tilde{x}}\tilde{u}(\tilde{x})$, so by Propsition \ref{rtp} (v),  there exists $A\in\mathcal{A}$ such that
\[D^2u\to A \quad \text{as} \ |x|\to+\infty.\]
Now, $|D^2u|\leq C(\theta,f,u)$ is bounded. Using Theorem \ref{wdecay} again, we obtain
\[D^2u(x)=A+O(|x|^{-\varepsilon})\quad \text{as}\ |x|\to +\infty, \]
where $\varepsilon\in (0,1)$ is a constant depending only on $\theta$, $\beta$ and $u$.

 Notice that Lemmas \ref{ab1}, \ref{ab2} and Steps 2--3 are also valid for $u$, and then we  shall continue to get the higher asymptotic behavior.

The equation for $w_1$ can be written as
\[\Delta w_1 =f_4:=f-\left(\tilde{a}_{ij}-\delta_{ij}\right)\left(w_1\right)_{ij}=O\left(|x|^{-\beta}\right) + O(|x|^{2\varepsilon_2 - 4}).\]
Using Lemma \ref{hmc} again, there exists $h_4$ such that
$\Delta h_4=f_4, |x| > 2^{2n_0+1}R_1$ and 
$$|h_4(x)| \leq C\left(|x|^{-\beta+2+\epsilon} +|x|^{2\varepsilon_2-2+\epsilon}\right), \quad |x| > 2^{2n_0+1}R_1$$
for $\epsilon > 0$ arbitrarily small. Since $w_1 - h_4$ is harmonic on $\mathbb{R}^2 \setminus B_{2^{2n_0+1}R_1}$ and $w_1 - h_4 = O(|x|^{\varepsilon_2})$. By Lemma \ref{hm2}, there exist  $d, c \in \mathbb{R}$ such that
$$w_{1}(x)-h_{4}(x)=d\ln|x|+c+O(|x|^{-1}).$$
Using the estimates on $h_{4}$ we have
\begin{equation}\label{abw1}
	w_{1}(x)=d\ln|x|+c+O(|x|^{2\varepsilon_{2}-2+\epsilon})\quad \text{as}\ |x|\to +\infty.
\end{equation}

\textbf{Step 4.}  Determining the $\frac{x}{|x|2}$ term.

Let
\[w_2(x)=w_1(x)-d\ln|x|-c.\]
By \eqref{abw1}, we already have
$$|w_2(x)| \leq C|x|^{-\varepsilon_3}, \quad |x| > 2^{2n_0+1}R_1$$
for some $\varepsilon_3 \in (0,2)$. Using Lemma \ref{ab1} as well as Schauder estimate, we obtain
$$
|D^k w_2(x)| \leq C|x|^{-\min \left\{\varepsilon_3,\beta-2\right\}-k}, \quad |x| > 2^{2n_0+2}R_1, \quad k=0,1,2.$$
Thus the equation of $w_2$ can be written as
\[\Delta w_2(x)=f_5:=f(x)-\left(\hat{a}_{ij}-\delta_{ij}\right)(w_2)_{ij}=O(|x|^{-\beta})+O\left(|x|^{-4-\varepsilon_4}\right),\]
where $\hat{a}_{ij}(x)=\int_0^1 F_{M_{ij}}\left(t\left(D^2w_2(x)+D^2(d\ln |x|)\right) + A\right)dt$ and constant $\varepsilon_4>0$. 

Using Lemmas 1 and 2 in \cite{LB2023}, there exists $h_5$ such that
$\Delta h_5=f_5$, $|x| > 2^{2n_0+2}R_1$ and
\[
h_5(x)=\begin{cases}
	O(|x|^{-\min \{\beta-2, 2+\varepsilon_4\}}), &\beta \neq 3,4,\\
	O(|x|^{2-\beta}\ln|x|), &\beta=3,4,
\end{cases}
\] 
as $|x|\to +\infty$. Since $w_2 - h_5$ is harmonic on $\mathbb{R}^2 \setminus B_{2^{2n_0+2}R_1}$ and $w_2 - h_5 = O(|x|^{-\varepsilon_5})$ for some $\varepsilon_5>0$.  By Lemma \ref{hm2}, we have 
$$w_{2}(x)-h_{5}(x)=O\left(|x|^{-1}\right).$$

Consequently, when $\beta>2$, we get
\[w_{2}(x)=\begin{cases}
	O(|x|^{2-\min \{\beta, 3\}}), &\beta\neq 3,\\
	O(|x|^{-1}\ln|x|), &\beta=3,
\end{cases}
\]
as $|x|\to +\infty$.

Moreover, if  $\beta >3$, 
Let $\tilde{w}_2(x) = w_2\left(\frac{x}{|x|^2}\right)$ and $\tilde{f}_5(x) = f_5\left(\frac{x}{|x|^2}\right)$ be the Kelvin transform of $w_2(x)$ and $f(x)$ respectively. Then we see
$$\tilde{w}_2(x) = O(|x|),\quad  x\in B_{\frac{1}{2^{2n_0+2}R_1}},  $$
and
$$
\Delta\tilde{w}_2(x) =f_6(x) := |x|^{-4}\tilde{f}_5(x) = \left(|x|^{\varepsilon_4}\right)+O(|x|^{\beta-4}),\quad  x \in B_{\frac{1}{2^{2n_0+2}R_1}}(0).$$
Then $f_6 \in L^{p_0}(B_{1/2^{2n_0+2}R_1}(0))$ for some ${p_0} > 2$, it follows that $\tilde{w}_2 \in W^{2,p_0}(B_{1/2^{2n_0+2}R_1}(0))$ and hence $\tilde{w}_2 \in C^{1,\zeta_0}(B_{1/2^{2n_0+2}R_1}(0))$ for $\zeta_0 = 1 - 2/p_0 \in (0,1)$. Then there exists $e \in \mathbb{R}^2$ and $\tilde{c} \in \mathbb{R}$ such that for some $C > 0$ such that
$$|\tilde{w}_2(x) - (e^\mathsf{T} x + \tilde{c})| \leq C|x|^{1+\zeta_0},\  x \in B_{\frac{1}{2^{2n_0+3}R_1}}(0).$$
Note that $\tilde{w}_2(0) = 0$ implies $\tilde{c} = 0$, we go back to exterior domain to get
\[\left| w_2(x) - \frac{e^\mathsf{T} x}{|x|^2} \right| \leq C |x|^{-1-\zeta_0}, \ x \in \mathbb{R}^2 \setminus B_{2^{2n_0+3}R_1}(0),\]
which leads to
\[u(x) = \frac{1}{2} x^\mathsf{T} A x + b^\mathsf{T} x + d \ln |x| + c + \frac{e^\mathsf{T} x}{|x|^2} + O(|x|^{-1-\zeta_0})\quad \text{as} \ |x|\to +\infty.\]

\textbf{Step 5.}  Calculating the value of $d$.

Let $Q(x) = \frac{1}{2} x^\mathsf{T} A x + b^\mathsf{T} x + c$. Then
\[u(x) = Q(x) + d \ln |x| + O(|x|^{-\min \{\beta-2, 1\}}(\ln |x|)^{\mu})\]
where $\mu$ is defined in Theorem \ref{D2ue}  and
\[\Delta (u - Q)(x) =\Delta w_2(x)=O(|x|^{-\beta})+O\left(|x|^{-4-\varepsilon_4}\right)\]
is integrable in $B_R \setminus \overline{\Omega}$. Let $\nu$ be the unit outward normal of boundaries $\partial \Omega$. Then by the divergence theorem, we have that for some $R > 0$ large enough and $\varepsilon_5\in (0,1]$,
\begin{align*}
	\int_{B_R \setminus \overline{\Omega}} \Delta(u - Q)(x) \, dx &= \int_{\partial(B_R \setminus \overline{\Omega})} \frac{\partial(u - Q)}{\partial \nu}  ds\\
	&=\int_{\partial B_R} (d \ln |x| + O(|x|^{-\varepsilon_5}))_\nu ds - \int_{\partial \Omega} (u - Q)_\nu  ds\\
	&=d\int_{\partial B_R} \frac{1}{R} ds+O\left(\frac{1}{R^{\varepsilon_5}}\right)-\int_{\partial \Omega} u_\nu ds + \int_{\Omega} \Delta Q  dx\\
	&=2\pi d + O\left(\frac{1}{R^{\varepsilon_5}}\right) - \int_{\partial \Omega} u_\nu  ds + \text{tr} A |\Omega|.
\end{align*}
Letting $R \to \infty$, we get 
\[d=\frac{1}{2\pi}\left(\int_{\partial \Omega} u_\nu  ds - \text{tr} A |\Omega|\right).\]

\textbf{Step 6.}  Higher order estimates of the error.

For $\beta>2$, bt Step 4, let
\[w_3(x)=u-\left(\frac{1}{2} x^\mathsf{T} A x + b^\mathsf{T} x + d \ln |x| + c \right).\]
and when $\beta>3$, let
\[w_4(x) = u - \left(\frac{1}{2} x^\mathsf{T} A x + b^\mathsf{T} x + d \ln |x| + c + \frac{e^\mathsf{T} x}{|x|^2}\right).\]
Then by Lemma \ref{ab1} and  Schauder estimate asserts that for all $k =0,\cdots,m+1$ and $|x| > 2^{2n_0+3}R_1$,
\[|D^kw_3(x)|\leq C(k)|x|^{-k-\min \left\{1, \beta-2\right\}}(\ln |x|)^{-\mu_1}.\]
where $\mu_1$ is defined in Theorem \ref{D2ue}.

Moreover, for  $\beta>3$, $k =0,\cdots,m+1$ and $|x| > 2^{2n_0+3}R_1$,
\[|D^kw_4(x)|\leq C(k)|x|^{-k-\zeta},\]
where $\zeta=1+\min \left\{\zeta_0, \beta-3\right\}$.

Hence the proof of Theorem \ref{D2ue} is complete under the assumption $F_{M_{ij}}(A)=\delta_{ij}$. If $F_{M_{ij}}(A)\neq\delta_{ij}$, let $P=\left(\left[\sqrt{F_{M_{ij}}(A)}\right]_{2\times2}\right)^{-1}=\sqrt{I+A^2}$ and $\hat{x}=Px.$ By the arguments above, we obtain the conclusion in $\hat{x}.$ Theorem \ref{D2ue} follows from transforming back to $x$.          \qed

\section{Construction of Subsolution and  Supersolution}\label{sec4}
We may assume without loss of generality that $E_1 \subset \subset\Omega \subset \subset E_{r_1}$ for some $1<r_1$. For any given $A\in \mathcal{A}_0$,  let $a := (a_1, a_2, \cdots, a_n) := \lambda(A)$ with $0 < a_1 \leq a_2 \leq \cdots \leq a_n$.

To get the exterior viscosity solution existence, we should first establish the sub and super solution of 
\begin{equation}\label{Lu3.1}
	F(D^2u) := f(\lambda(D^2u)) := \sum_{i=1}^n \arctan \lambda_i(D^2u) = \theta+f(x),
\end{equation}
 in $\mathbb{R}^n\setminus \Omega$, and $0<2\delta<\theta+f(x)<n\pi/2-2\delta$. 

\subsection{Generalized symmetric subsolutions}\label{gssub}
We will construct subsolutions \( \underline{w}(x)=\phi(r) \) (see Section \ref{sec2.1}) that are generalized symmetric with respect to \( A \)
by solving initial value problems of ordinary differential equations  satisfied
by \( \phi \in C^2([1, +\infty)) \). It implies that \( \underline{u} \in C^2(\mathbb{R}^n \setminus \Omega) \) are classical subsolutions.

By condition \eqref{abf}, there exist monotone increasing smooth function \( \underline{f}(r) \) and monotone decreasing smooth function \( \overline{f}(r) \) on \([1, +\infty)\) such that
\begin{equation}\label{uof}
	\theta-\delta \leq  \underline{f}(r(x)) \leq \theta+f(x) \leq \overline{f}(r(x))\leq 	\theta+\delta, \quad \forall \, x \in \mathbb{R}^n \setminus E_1
\end{equation}
and
\[
\underline{f}(r), \, \overline{f}(r) = \theta + O_m(r^{-\beta}) \quad \text{as } r \to \infty.
\]

By the monotonicity of arctan function, there exists a unique decreasing positive function \( \underline{h}(r) \) defined on \([1, +\infty)\) determined by
\begin{equation}\label{ovrh}
	g \left( a_1\underline{h}(r), \cdots, a_n\underline{h}(r) \right) = \overline{f}(r),
\end{equation}
where $g(\lambda_1,\cdots,\lambda_n)=\sum_{i=1}^n \arctan \lambda_i$. Since $A\in \mathcal{A}_0$, then $ \underline{h}(r)\geq 1 $ and $\lim\limits_{r\to \infty}\underline{h}(r)=1$.

Next, we give the subsolutions existence result.  
\begin{prop}\label{sublem}
Let $f$ satisfying \eqref{uof} with $\beta>2$. Let $A \in \mathcal{A}_{0}$, $\alpha_1\in \mathbb{R}$ and $\beta_1>\underline{h}(1)$. 
  Then equation \eqref{Lu3.1} admits a  subsolution $\underline{w}_{\alpha_1,\beta_1} \in C^2(\mathbb{R}^n\setminus \Omega)$. This subsolution is generalized symmetric with respect to $A$ and has the asymptotic behavior
	\begin{equation}\label{omega1}
		\underline{w}_{\alpha_1,\beta_1}(x) = \frac{1}{2}x^{\mathsf{T}}Ax + c_{\alpha_1,\beta_1} + \begin{cases} 
			O(r^{2-\min\{d(A,\epsilon_0), \beta\}}), & \text{if } d(A,\epsilon_0) \neq \beta, \\ 
			O(r^{2-\beta} \ln r), & \text{if } d(A,\epsilon_0) = \beta, 
		\end{cases}  
	\end{equation}
as $|x| \to +\infty$, where $c_{\alpha_1,\beta_1}$ is a constant depending on $n, A, b,\theta, f, \alpha_1$ and $\beta_1$. Moreover, if $\theta+f(x)>(n-2)\pi/2$,   the asymptotic behavior becomes
	\begin{equation}\label{omega1sc}
	\underline{w}_{\alpha_1,\beta_1}(x) = \frac{1}{2}x^{\mathsf{T}}Ax + c_{\alpha_1,\beta_1} + \begin{cases} 
		O(r^{2-\min\{d(A,0), \beta\}}), & \text{if } d(A,0) \neq \beta, \\ 
		O(r^{2-\beta} \ln r), & \text{if } d(A,0) = \beta, 
	\end{cases}  
\end{equation}
as $|x| \to +\infty$.
$\underline{w}_{\alpha_1,\beta_1}$ have  limit
\[\lim\limits_{\alpha_1\to +\infty}\underline{w}_{\alpha_1,\beta_1}(x)=+\infty,\quad \text{and}\quad \lim\limits_{\beta_1\to +\infty}\underline{w}_{\alpha_1,\beta_1}(x)=+\infty,\ \forall x\in \mathbb{R}^n\setminus \Omega. \]
 Here, $d(A,\epsilon_0)$ with $\epsilon_0\geq 0$ is defined as below.
\end{prop}

\begin{rem}\label{MAcond}
When $A$ is diagonal,	 $d(A)$ in \eqref{MAeps}  becomes
	\begin{equation*}
		d(A) =\frac{1+a_1^2}{a_n} \cdot \sum_{i=1}^n\frac{a_i}{1+a_i^2}.
	\end{equation*}
	Moreover, for any $A\in \mathcal{A}_0$, there exists  constant $\epsilon_0> 0$ such that
	\begin{equation}\label{MAeps1}
		d(A,\epsilon_0) := \left(\frac{a_n}{1+a_1^2}+\epsilon_0\sum_{i=2}^n\frac{a_{i}}{1+a_i^2} \right)^{-1}\cdot \sum_{i=1}^n\frac{a_i}{1+a_i^2} > 2.
	\end{equation}
\end{rem}

 Recall formular \eqref{gra},  the Laplacian of $\underline{w}$ satisfying
\begin{equation}\label{deltaomg}
	\Delta \underline{w} = \left(\sum_{i=1}^{n}a_i\right)h + r^{-1}h'\sum_{j=1}^{n}a_j^2x_j^2,
\end{equation}
where $h(r)=\underline{w}^{\prime}(r)/r$.

Applying Weyl's eigenvalue theorem, we obtain the following estimates for the Hessian eigenvalues of generalized symmetric functions (cf. Lemmas 2.1, 2.3 in \cite{LW2024} and Lemma 11 in \cite{BLW2024}):

\begin{lem}\label{evest}
	For any $C^2$ generalized symmetric function $\underline{w}$ with respect to $A$, assume $h > 0$ and $h' \leq 0$, the eigenvalues satisfy:
	\begin{equation}\label{evest1}
		a_ih + r^{-1}h'\sum_{j=1}^{n}a_j^2x_j^2 \leq \lambda_i(D^2\underline{w}) \leq a_ih, \quad \forall i=1,\cdots,n.
	\end{equation}
	Moreover, 
	\begin{equation}\label{evest2}
		F(D^2\underline{w}) \geq g(\overline{a}_{\epsilon_0}),
	\end{equation}
	 provided $a_1h + a_nrh' \geq 0$. Here $\overline{a}_{\epsilon_0} := \big(a_1h + a_nrh', a_2\left(h +\epsilon_0rh'\right), \cdots, a_n\left(h +\epsilon_0rh'\right)\big)$.
\end{lem}
\begin{proof}
	Estimate \eqref{evest1} can be obtained as Lemma 2.1 in \cite{LW2024}.	
	By \eqref{deltaomg}, there exists $0\leq \theta_i(r)\leq 1,i=1,\cdots,n$, such that 
	\begin{equation}\label{lmduw}
		\lambda_i(D^2\underline{w})=a_ih+\theta_ir^{-1}h^{\prime}\sum_{j=1}^{n}a_j^2x_j^2\quad\text{and}\quad\sum_{i=1}^n\theta_i=1.
	\end{equation} 
	Let $V:=r^{-1}h^{\prime}\sum_{j=1}^{n}a_j^2x_j^2$, 	Then, for any $r\geq 1$,
	\begin{equation}\label{ineqV}
		a_nrh^{\prime}\leq V\leq a_1rh^{\prime}\leq 0.
	\end{equation}
	Using this fact and combine the Lemma 4 in \cite{BLW2024}, by the monotonicity of  $\arctan$, we have
	\begin{align*}
		F(D^2\underline{w})&=g(a_1h+\theta_1V,\cdots,a_nh+\theta_nV)\\
		&\geq g(a_1h+V,a_2h,\cdots,a_nh)\\
		&\geq g(a_1h+a_nrh^{\prime},a_2\left(h +\epsilon_0rh'\right), \cdots, a_n\left(h +\epsilon_0rh'\right)),
	\end{align*}
	we finish the proof.
\end{proof}

\begin{lem}\label{Ipthm}
Suppose $\epsilon_0>0$.	There exists a unique nonpositive smooth function $I(r,h)$ satisfying 
	\begin{equation}\label{imIeq}
		g(a_1h+a_nI,a_2h+a_{2}\epsilon_0I,\cdots,a_nh+a_{n}\epsilon_0I)=\overline{f}(r)
	\end{equation}
	in $\left\{(r,h) \vert r\geq 1, \beta_1\geq h>\underline{h}(r) \right\}$.
	Especially, $I(r,h)=0$ if and only if $h(r)=\underline{h}(r)$,  and $I(r,h)$  is monotone decreasing with respect to  both $r$ and $h$. Furthermore, there exists $C>0$ such that
	\begin{equation}\label{eqI}
		|I(r,h)-I(\infty,h)|\leq Cr^{-\beta}\  \text{and}\ \ \frac{\partial I}{\partial h}(\infty,1)=-d(A,\epsilon_0).
	\end{equation}
\end{lem}
\begin{proof}
	For any $r\geq 1$ and  $\underline{h}(r)<h$, we have
	\begin{align*}
		&\quad \lim_{I\to 0}g(a_1h+a_nI,a_2h+a_{2}\epsilon_0I,\cdots,a_nh+a_{n}\epsilon_0I)\\
		&=g(a_1h,a_2h,\cdots,a_nh)> \overline{f}(r).
	\end{align*}
	On the other hand, $\overline{f}(r)\geq \theta$, and it follows the monotonicity of $g$ that
	\begin{align*}
		&\quad \lim_{I\to m_{\epsilon_0,\beta_1}}g(a_1h+a_nI,a_2h+a_{2}\epsilon_0I,\cdots,a_nh+a_{n}\epsilon_0I)\\
		&\leq g\left(a_1\left(h-\beta_1\right)+a_1,a_2\left(h-\beta_1\right)+a_2,\cdots,a_n\left(h-\beta_1\right)+a_n\right)\\
		&\leq  g\left(a_1,a_2,\cdots,a_n\right)=\theta<\overline{f}(r),
	\end{align*}
where $m_{\epsilon_0,\beta_1}=\min\left\{\frac{1-\beta_1}{\epsilon_0},\frac{a_1(1-\beta_1)}{a_n}\right\}$.	Hence, by the mean value theorem and the implicit function theorem, there exists a unique function $m_{\epsilon_0,\beta_1}<I(r,h)<0 $ such that \eqref{imIeq} holds and it is a smooth, bounded function
	with respect to $r$ and $h$.
	
	Furthermore, taking partial derivative with respect to $r$ and $h$, we have
	and for any $r\geq 1,   \underline{h}(r)\leq h\leq \beta_1$,
	\begin{align*}
		0>\overline{f}^{\prime}(r) &=\left(a_n\frac{\partial g}{\partial\lambda_1}(\tilde{\lambda}) +\epsilon_0\sum_{i=2}^n a_{i}\frac{\partial g}{\partial\lambda_i}(\tilde{\lambda})\right)\cdot\frac{\partial I}{\partial r}(r,h),
	\end{align*}
	and 
	\begin{align*}
		0&=\sum_{i=1}^n\frac{\partial g}{\partial\lambda_i}(\tilde{\lambda})a_i+\left(a_n\frac{\partial g}{\partial\lambda_1}(\tilde{\lambda}) +\epsilon_0\sum_{i=2}^n a_{i} \frac{\partial g}{\partial\lambda_i}(\tilde{\lambda})\right)\cdot\frac{\partial I}{\partial h}(r,h),
	\end{align*}
	where $\tilde{\lambda}:=(a_1h+a_nI,a_2h+a_{2}\epsilon_0I,\cdots,a_nh+a_{n}\epsilon_0I)$.
	Hence, $I(r,h)$ is monotone decreasing with respect to  both $r$ and $h$. Sending $(r,h)\to (\infty,1)$,  a simple arguement shows $I(\infty,1)=0$, since $I(r,h)<0$.  It follows immediately that
	\[0=\sum_{i=1}^n \frac{a_i}{1+a_i^2}+\left(\frac{a_n}{1+a_1^2}+\epsilon_0\sum_{i=2}^n\frac{a_{i}}{1+a_i^2}\right)\frac{\partial I}{\partial h}(\infty,1),\]
	that is the second equality in \eqref{eqI}.

	When $h=\underline{h}(r)$, by the monotonicity of $g$ and \eqref{ovrh}, we have
	\begin{align*}
		\overline{f}(r)=g(a_1\underline{h},a_2\underline{h},\cdots,a_n\underline{h}), 
	\end{align*}
	and
	\begin{align*}
		\overline{f}(r)=g(a_1\underline{h}+a_nI(r,\underline{h}),a_2\underline{h}+a_{2}\epsilon_0I(r,\underline{h}),\cdots,a_n\underline{h}+a_{n}\epsilon_0I(r,\underline{h})).
	\end{align*}
	since $I(r,\underline{h})\leq 0$, $a_i>0$, the only possibility to make
	the equalities above hold is $I(r,\underline{h})=0$. It's clear that $I(r,h)<0$ when $\underline{h}(r)<h$.

	Finally, we prove the asymptotic behavior of $I$. Since $I(r,h)$ is bounded, monotone increasing with
	respect to $r$, there exists $C>0$ such that
	\begin{align*}
		\overline{f}(r)-\theta&\geq \arctan\left(a_1+a_nI(r,h)\right)-\arctan\left(a_1+a_nI(\infty,h)\right)\\
		&=\frac{a_n}{1+\xi^2}\left(I(r,h)-I(\infty,h)\right)\\
		&\geq C\left(I(r,h)-I(\infty,h)\right),
	\end{align*}
	where $\xi\in \left(a_1h+a_nI(\infty,h),a_1h+a_nI(r,h)\right)$ is bounded. This finishes the proof of this lemma.
\end{proof}

\begin{rem}\label{scc}
	If $\theta+f(x)>(n-2)\pi/2$, we can take $\epsilon_0=0$, the result in the lemma \ref{Ipthm} still holds (see \cite[Lemma 6]{BLW2024}). In this case, the right hand of the second inequality in  \eqref{eqI}  becomes $-d(A,0)$.
\end{rem}

\begin{cor}\label{subst}
	Let $\underline{w}, h, A$ and $I$ be as in Lemmas \ref{evest} and \ref{Ipthm}. If $h\in C^{1} [1,+\infty)$ satisfying
	\[\beta_1\geq h>\underline{h}(r),\quad h^{\prime}\leq 0\quad \text{and} \quad  h^{\prime}\geq \frac{I_{0}(r,h)}{r}\ \text{in}\ (1,\infty),\]
where $I_{0}(r,h)=\max\left\{I(r,h),-a_1h/a_n\right\}$.	Then, $\underline{w} \in C^{2}(\mathbb{R}^n\setminus \Omega)$ is a subsolution to \eqref{Lu3.1}.
\end{cor}
\begin{proof}
	From the definition of generalized symmetric functions as in \eqref{gra}, $h\in C^{1}[1,+\infty)$, which implies $\underline{w}\in C^{2}(\mathbb{R}^n\setminus E_{1})$ satisfies the following differential inequality
	\begin{align*}
		F(D^2\underline{w})&\geq g\left(a_1h+a_nrh^{\prime},a_2h+a_{2}\epsilon_0rh^{\prime},\cdots,a_nh+a_{n}\epsilon_0rh^{\prime}\right)\\
		&\geq g(a_1h+a_nI_0,a_2h+a_{2}\epsilon_0I_0,\cdots,a_nh+a_{n}\epsilon_0I_0)\\
		&\geq g(a_1h+a_nI,a_2h+a_{2}\epsilon_0I,\cdots,a_nh+a_{n}\epsilon_0I)\\
		&=\overline{f}(r)\geq f(x) \  \text{in}~~ \mathbb{R}^n\setminus E_{1}.
	\end{align*}
	 Here, the first inequlity we use Lemma \ref{evest} and the last euality follows Lemma \ref{Ipthm}.
	Since  $E_{1}\subset\subset \Omega$, this finishes the proof.
\end{proof}

Next, we turn to construct solution satisfied the assumption in Corollary \ref{subst}.
\begin{lem}\label{csth}
For any $\beta_1\geq \underline{h}(1)$,	there exists a unique  solution $h_1(r;\beta_1)$ to
	\begin{equation}\label{dh}
		\left.\left\{\begin{array}{ll}
			\frac{dh_1}{dr}=\frac{I_0(r,h_1)}{r},& r>1,\\
			h_1(1)=\beta_1.
		\end{array}\right.
		\right.
	\end{equation}
	Futhermore, 
	\begin{enumerate}
		\item[(i)] $1\leq h_1(r;\beta_1)\leq \beta_1$, $\partial h_1/\partial r\leq 0$. More specifically, $h_1(r;\underline{h}(1))=\underline{h}(r)$ and $1<h_1(r;\beta_1)< \beta_1$, $\forall r>1, \beta_1>\underline{h}(1)$.
		\item[(ii)]  $h_1$ has asymptotic
		behavior \begin{equation}\label{abh}
			h_1(r) = 1 + 
			\begin{cases} 
				O(r^{-\min\{d(A,\epsilon_0), \beta\}}), & \text{if } d(A,\epsilon_0) \neq \beta, \\ 
				O(r^{-\beta} \ln r), & \text{if } d(A,\epsilon_0) = \beta, 
			\end{cases}
		\end{equation}
		as $r\to \infty$, where $d(A,\epsilon_0)$ is defined in \eqref{MAcond}. Moreover, if $\theta+f(x)>(n-2)\pi/2$,   
		$h_1$ has asymptotic
		behavior \begin{equation}\label{abhsc}
			h_1(r) = 1 + 
			\begin{cases} 
				O(r^{-\min\{d(A,0), \beta\}}), & \text{if } d(A,0) \neq \beta, \\ 
				O(r^{-\beta} \ln r), & \text{if } d(A,0) = \beta, 
			\end{cases}
		\end{equation}
		as $r\to \infty$. 
	\item[(iii)] $h_1(r;\beta_1)$ is continuous and strictly increasing with respect to $\beta_1$ and  
	\begin{equation}
		\label{hab}
		\lim\limits_{\beta_1\to +\infty} h_1(r;\beta_1)=+\infty,\quad \forall r\geq1.
	\end{equation}
	\end{enumerate}
\end{lem}
\begin{proof}
	The proof of this lemma now will be divided into three
	steps.
	
	\textbf{Step 1.}  Proof of (i).
	
	Since $\beta_1\geq\underline{h}(1)$, and $\underline{h}$ is monotone decreasing with respect $r$, there exists $r_0>1$ such that
	\[[1,1+r_0]\times \left[\frac{\beta_1+\underline{h}(1)}{2},\frac{2\beta_1+\underline{h}(1)}{3}\right]\subset\left\{(r,h_1):r\geq 1,\beta_1\geq h_1\geq \underline{h}(r)\right\}.\] 
	Since $I$ is smooth and bounded, the right side in \eqref{dh} is Lipschtiz in
	the rectangle $[1,1+r_0]\times [(\beta_1+\underline{h}(1))/2,(2\beta_1+\underline{h}(1))/3]$. By the existence and uniqueness theorem of ODE,	the initial value problem \eqref{dh} admits
	locally a unique smooth solution $h_1$ near $r=1$ and $h_1^{\prime}(1)=I_0(1,\beta_1)$. We shall prove that the solution can be
	extended to $r=[1,+\infty)$.
	
	By Lemma \ref{Ipthm} and the fact that $I$ is nonpositive, and then $I_0$ is nonpositive, hence $h_1$ is monotone decreasing as long as $h_1>\underline{h}(r)$. Next, if $\beta_1>\underline{h}(1)$, we claim $h_1(r)>\underline{h}(r)$ for all $r\geq 1$. Arguing by contradiction, if there exists $r_1>1$ such that
	\[h_1(r_1)=\underline{h}(r_1),\quad h_1(r)>\underline{h}(r),~~\forall r\in [1,1+r_1).\]
	Then, a simple calculation shows 
	\[h_1^{\prime}(r_1)=\lim_{r\to r_1^-}\frac{h_1(r)-h_1(r_1)}{r-r_1}\leq \underline{h}^{\prime}(r_1),\]
	but it follows Lemma \ref{Ipthm} that $I(r,\underline{h}(r))=0$ and $h_1\geq \underline{h}\geq 1$, then
	\[h_1^{\prime}(r_1)=  \frac{I_0(r_1,\underline{h}(r_1))}{r_1}=\frac{I(r_1,\underline{h}(r_1))}{r_1}= 0.\]
	A contradiction since $\underline{h}$ is strictly monotone decreasing. Combining
	the results above, $h_1$ is monotone decreasing and $\beta_1\geq h_1(r)>\underline{h}(r)\geq 1$ holds as long as $h_1(r)$
	exists. By the Carathéodory extension theorem of ODE, the smooth solution $h_1$ exists for all $r\geq 1$.
	
	If \(\beta_1 = \underline{h}(1)\), then \(h_1=\underline{h}\) is a solution of the problem \eqref{dh},
	since \(I_0(r,\underline{h}(r))=I(r,\underline{h}(r))=0\) according to Lemma \ref{Ipthm} and Corollary \ref{subst}. Thus, by the uniqueness theorem for the solution
	of the ordinary differential equation, we know that \(h_1(r, \underline{h}(1)) = \underline{h}\) is the unique
	solution satisfying the problem \eqref{dh}.

	\textbf{Step 2.}  Proof of (ii).
	
	Since $h_1$ is positive, monotone
	decreasing and bounded from below by $\underline{h}(r)\geq 1$, hence $h_1$ admits a finite limit $\beta_1> h_1(\infty)\geq 1$, and we claim that $\lim \sup_{r\to \infty}I_0(r,h_1(r))=0$. Arguing by contradiction, suppose $-2\varepsilon:=\lim \sup_{r\to \infty}I_0(r,h_1(r))<0$, then there exists $\overline{r}$  sufficiently large such that
	\[\frac{I_0(r,h_1(r))}{r}\leq \frac{-\varepsilon}{r}, \quad \forall r\geq \overline{r},\]
	which contradicts the Newton--Leibniz formula
	\[h_1(\infty)-h_1(\overline{r})=\int_{\overline{r}}^{+\infty}\frac{I_0(r,h_1)}{r}dr\leq -\varepsilon \int_{\overline{r}}^{+\infty}\frac{1}{r}dr=-\infty.\]
	Furthermore,we claim tha $h_1(\infty)=1$.   Arguing by contradiction, if $h_1(\infty)>1$, then by Lemma \ref{Ipthm} and there there exists a subsequence $\left\{r_j\right\}_{j=1}^{\infty}$ such that,
	\[\lim_{j\to \infty} r_j=+\infty,\quad \text{and} \quad  \lim_{j\to \infty} I_0(r_j,h_1(r_j))=0.\]
	Since $h_1(r_j)\geq 1$, we get  $\lim_{j\to \infty} I(r_j,h_1(r_j))=0$, and then 
	\begin{align*}
		\theta&=\lim_{j\to \infty}	\overline{f}(r_j)\\
		&=\arctan \left(a_1h_1(\infty)+a_n\cdot0\right) +\sum_{i=2}^n \arctan \left(a_ih_1(\infty)+a_{i}\epsilon_0\cdot0\right)\\
		&=\sum_{i=1}^n \arctan a_ih_1(\infty) >\sum_{i=1}^n \arctan a_i=\theta, 
	\end{align*}
	which is a contradiction. Thus, we obtain $\lim_{r\to \infty} I(r,h_1(r))=0$. Indeed, we only to prove $\lim \inf_{r\to \infty}I(r,h_1(r))=0$. Arguing by contradiction, suppose $-\varepsilon:=\lim \inf_{r\to \infty}I(r,h_1(r))<0$. 
	
	Therefore,  there exists a subsequence $\left\{r_j\right\}_{j=1}^{\infty}$ such that,
	\[\lim_{j\to \infty} r_j=+\infty,\quad \lim_{j\to \infty}h_1(r_j)=1\quad  \text{and} \quad  \lim_{j\to \infty} I(r_j,h_1(r_j))=-\varepsilon,\]
   which yields
	\begin{align*}
		\theta&=\lim_{j\to \infty}	\overline{f}(r_j)=\arctan \left(a_1-a_n\varepsilon\right) +\sum_{i=2}^n \arctan \left(a_ih_1(\infty)-a_{i}\epsilon_0\varepsilon\right)\\
		&<\sum_{i=1}^n \arctan a_i =\theta, 
	\end{align*}
	which is a contradiction.

	Next, we refine the asymptotic behavior by setting
	\[t:=\ln r\in [0,+\infty) \quad \text{and}\quad \phi(t)=h_1(r(t))-1.\]
	By a direct computation, for all $t\in [0,+\infty)$
	\begin{align*}
		\phi^{\prime}(t)=h_1^{\prime}(r(t))\cdot e^t&=I_0(r(t),\phi+1)\\
		&=I_0(r(t),\phi+1)-I_{0}(\infty,\phi+1)+I_0(\infty,\phi+1)\\
		&=:I_1(t,\phi)+I_2(\phi).
	\end{align*}
	In view of  \eqref{eqI}, there exists $C>0$ such that for all $t\gg1$ and $\phi\ll 1$,  $I_{0}=I$ and 
	\[|I_{1}(t,\phi)|\leq Ce^{-\beta t},\quad \frac{dI_2}{d\phi}(0)=\frac{dI}{dh}(\infty,1)=-d(A,\epsilon_0),\]
	and 
	\[\left|I_2(\phi)-\frac{dI_2}{d\phi}(0)\phi\right|\leq C\phi^2.\]
	Consequently, $\phi$ satisfies
	\begin{equation}\label{dphi}
		\phi^{\prime}(t)=-d(A,\epsilon_0)\phi+O(e^{-\beta t})+O(\phi^2)
	\end{equation}
	as $t \to \infty$ and $\phi \to 0$.
	since $\phi> 0$, for any sufficiently small $\epsilon>0$, $\phi $ satisfies
	\[\phi^{\prime}\leq -(d(A,\epsilon_0)-\epsilon)\phi+Ce^{-\beta t}, \quad \forall t>T_0,\]
	for some large $C$ and $T_0$. Multiplying both sides by $e^{(d(A,\epsilon_0)-\epsilon)t}$ and taking
	integral over $(T_0,t)$, there exists $C>0$ such that
	\begin{equation*}
		\phi(t) \leq 
	\begin{cases} 
		Ce^{-\min\{d(A,\epsilon_0)-\epsilon, \beta\} t}, & \text{if } d(A,\epsilon_0) - \epsilon \neq \beta, \\ 
		Cte^{-\min\{d(A,\epsilon_0)-\epsilon, \beta\} t}, & \text{if } d(A,\epsilon_0) - \epsilon = \beta. 
	\end{cases}
	\end{equation*}
 Taking this estimate into equation \eqref{dphi} and choosing $\epsilon$ sufficiently small 
	we have
	\begin{equation*}
		\phi^{\prime} \leq -d(A,\epsilon_0)\phi+Ce^{-\beta t}+\begin{cases} 
			Ce^{-2\min\{d(A,\epsilon_0)-\epsilon, \beta\} t}, & \text{if } d(A,\epsilon_0) - \epsilon \neq \beta, \\ 
			Ct^2e^{-2\min\{d(A,\epsilon_0)-\epsilon, \beta\} t}, & \text{if } d(A,\epsilon_0) - \epsilon = \beta, 
		\end{cases} \ \forall t>T_0^{\prime},
	\end{equation*}
	for some large $C$ and $T_0^{\prime}>T_0$. Multiplying both sides by $e^{d(A,\epsilon_0)t}$ and taking integral
	over $(T_0^{\prime},t)$, we have the desired estimate \eqref{abh} for one side, the proof of the converse inequality is similar.  \eqref{abhsc} follows form Remark \eqref{scc}.

		\textbf{Step 3.}  Proof of (iii).
		
		We first prove that $h_1$ is strictly increasing with respect to initial values for $I_0(r,h_1)=I(r,h_1)$ and $I_0(r,h_1)=-a_1h_1/a_n$, respectively.
		
		Case 1. $I_0(r,h_1)=-a_1h_1/a_n$ and 
		\begin{equation*}
			\left.\left\{\begin{array}{ll}
				\frac{dh}{dr}=-\frac{a_1h_1}{a_nr},& r_1<r<r_1+\delta_1,\\
				h_1(r_1)=c_0,
			\end{array}\right.
			\right.
		\end{equation*}
		for some $r_1\geq 1$, $c_0>1$ and  $\delta_1>0$. It's easy to see the solution of this ODE is
		\begin{equation}\label{h1eq}
			h_1(r)=c_0\left(\frac{r_1}{r}\right)^{\frac{a_1}{a_n}},\quad r_1\leq r<r_1+\delta_1.
		\end{equation}
		Consequently, $h_1$ is strictly increasing with respect to $c_0$.
		
		Case 2. $I_0(r,h_1)=I(r,h_1)$ and
		\begin{equation*}
			\left.\left\{\begin{array}{ll}
				\frac{dh}{dr}=-\frac{I(r,h_1)}{r},& r_2<r<r_2+\delta_2,\\
				h_1(r_2)=c_1,
			\end{array}\right.
			\right.
		\end{equation*}
		for some $r_2\geq 1$, $c_1>1$ and  $\delta_2>0$.
		
			By the theorem of the differentiability of the solution with respect
		to the initial value, we can differentiate $h_1(r;c_1)$ with respect to $h_2$ as below:
		\[
		\begin{cases} 
			\frac{\partial v}{\partial r} = \frac{\partial I(r, h_1(r;c_1))/\partial h_1}{r}\cdot v , \\ 
			v(1) = 1, 
		\end{cases}
		\]
		where $v(r) := \frac{\partial h_1(r;c_1)}{\partial c_1}$. Therefore, we get
		\[v(r)=\exp\left(\int_1^r\frac{\partial I(\tau, h_1(\tau;c_1))/\partial h_1}{\tau}d\tau\right).\]
		By Lemma \ref{Ipthm} and \eqref{eqI}, $\frac{\partial I(\tau; h_1(\tau,\beta_1))}{\partial h_1}$ is bounded  and 
			\[\lim_{r\to \infty} \frac{\partial I(\tau, h_1(\tau;c_1))}{\partial h_1}=-d(A,\epsilon_0), \]
and hence $0<\partial h_1(r;c_1)/\partial c_1<r^{C},\ \forall r_2\leq r<r_2+\delta_2.$
Consequently,  \(h_1(r;c_1)\) is strictly increasing with respect to \(c_1\).

For any $r>1$, there exist small intervals $\left\{I_k \right\}_{k}$ such that $\sum_{k}|I_k|\geq r-1$ and $I_0(r,h_1)$ is equal to either $I(r,h_1)$ or $-a_1h_1/a_n$ on each interval. Now, since $h_1$ is strictly increasing with respect to initial values either $I(r,h_1)$ or $-a_1h_1/a_n$, we conclude the result by inductive reasoning. 

Finally, we  prove \eqref{hab}
by contradiction. Suppose not. There would exist \(r_0 \geq 1\), \(M > 1\) and \(\{(\beta_1)_k\}_{k=1}^\infty\), \(1 < (\beta_1)_k \to +\infty\) (\(k \to +\infty\)) such that \(h_1(r_0; (\beta_1)_k) \leq M\), \(\forall k \in \mathbb{Z}_+\). Note that there are infinitely many \((\beta_1)_k > M\) satisfying \(1 \leq h_1(r_0; (\beta_1)_k) \leq M < (\beta_1)_k\). Since
\[
\frac{dh}{dr} = \frac{I_0(r,h_1)}{r},
\]
satisfies \(I_0(r,\underline{h}(r)) = 0\) and \(0 < -I_0(r,h_1) < -a_1h_1/a_n (\forall h_1 > 1)\) , we have
\[
\frac{a_n dh_1}{a_1h_1} \leq \frac{dh_1}{-I_0(r,h_1)} = -\frac{dr}{r}.
\]
Integrating it from \(M\) to \((\beta_1)_k\) and recalling \(h_1(1, (\beta_1)_k) = (\beta_1)_k\), we get
\[
\int_M^{(\beta_1)_k} \frac{a_ndh_1}{a_1h_1} \leq \int_M^{(\beta_1)_k} \frac{dh_1}{-I_0(r,h_1)} \leq \int_{(h_1(r_0; (\beta_1)_k)}^{(\beta_1)_k} \frac{dh_1}{-I_0(r,h_1)} = -\int_{r_0}^1 \frac{dr}{r} = \ln r_0 < +\infty.
\]
Let \((\beta_1)_k \to +\infty\), we have
\[
\int_M^{(\beta_1)_k} \frac{a_ndh_1}{a_1h_1} \to +\infty,
\]
which is a contradiction. Hence the assertion (iii) of the lemma is proved.
\end{proof}

\noindent\textbf{Proof of Propsition \ref{sublem}:}  For any $\alpha_1\in \mathbb{R}$ and $\beta_1>\underline{h}(1)$,  let
\begin{equation}\label{w}
	\begin{aligned}
		\underline{w}_{\alpha_1,\beta_1}(x)=\phi_{_{\alpha_1,\beta_1}}(r)&=\alpha_1+\int_{1}^rsh_1(s;\beta_1)ds\\
		&=:\alpha_1+\frac{1}{2}\sum_{i=1}^na_ix_i^2+\mu_1(\beta_1)-\int_{r}^{\infty}s(h_1(s;\beta_1)-1)ds,
	\end{aligned}
\end{equation}
where 
\[\mu_1(\beta_1)=\int_{1}^{\infty}s(h_1(s;\beta_1)-1)ds-\frac{1}{2}.\]

Combine the above results with Corollary \ref{subst}, $\underline{w}$ is a  subsolution of \eqref{Lu3.1} with initial date $\underline{w}_{\alpha_1,\beta_1}(1)=\alpha_1$ and $\underline{w}_{\alpha_1,\beta_1}^{'}(1)=\beta_1$.

By \eqref{abh}, it's easy to see when $r\to \infty$, 
\begin{equation*}
	\int_{r}^{\infty}s(h_1(s;\beta_1)-1)ds=
\begin{cases} 
	O(r^{2-\min\{d(A,\epsilon_0), \beta\}}), & \text{if } d(A,\epsilon_0) \neq \beta, \\ 
	O(r^{2-\beta} \ln r), & \text{if } d(A,\epsilon_0) = \beta, 
\end{cases}
\end{equation*}
In view of  Remark \ref{scc},  if $\theta+f(x)>(n-2)\pi/2$, asymptotic  behavior becomes
 \begin{equation*}
	\int_{r}^{\infty}s(h_1(s;\beta_1)-1)ds=
	\begin{cases} 
		O(r^{2-\min\{d(A,0), \beta\}}), & \text{if } d(A,0) \neq \beta, \\ 
		O(r^{2-\beta} \ln r), & \text{if } d(A,0) = \beta, 
	\end{cases}
\end{equation*}
as $r\to \infty$.
Inserting the above equality into \eqref{w}, we can obtain the asymptotic behavior \eqref{omega1} and \eqref{omega1sc}.

Following  Lemma \ref{csth} (iii), we observe that the function $\mu_1(\beta_1)$ is strictly increasing in $\beta_1$ and $\mu_1(\beta_1)\to +\infty$ as $\beta_1\to +\infty$. Consequently, the range of $\mu_1$ is precisely $(\mu_1(\underline{h}(1)), +\infty)$.  \qed

\subsection{Generalized symmetric supersolutions}\label{gssup}
If $f\geq  0$, then we can take $\overline{u} = \frac{1}{2}x^{\mathsf{T}} Ax + c$ directly as the desired supersolution, provided that c > $c^*$ and $c^*$ is selected suitably, see for instance the argument in \cite{CL2003,BLL2014,LZ2019}.
Otherwise, 
in order to prove Theorem \ref{tedp} in Sec. \ref{sec5}, we still need to find supersolutions on  $\mathbb{R}^n\setminus \Omega$.

Now, we turn to construct supsolutions \( \overline{w}(x)=\phi(r) \)  that are generalized symmetric with respect to \( A \).
   By the discussion in \eqref{uof}, there exists a unique increasing positive function \( \overline{h}(r) \) defined on \([1, +\infty)\) determined by
\begin{equation}\label{ovrh1}
	g \left( a_1\overline{h}(r), \cdots, a_n\overline{h}(r) \right) = \underline{f}(r).
\end{equation}
 Recall that $A\in \mathcal{A}_0$, and then $ \overline{h}(r)\leq 1 $ and $\lim\limits_{r\to \infty}\overline{h}(r)=1$. 
 
 Next, we give the supersolutions existence Lemma.  
 \begin{prop}\label{suplem}
Let $f$ satisfying \eqref{uof} with $\beta>2$.  Let $A \in \mathcal{A}_{0}$,   $\alpha_2\in \mathbb{R}$   and $0<\beta_2<\overline{h}(1)$. 
 Then	there exists a constant $\overline{r}=\overline{r}(n,A,b,\epsilon_0,f)$, such that  exterior problem \eqref{Lu3.1} admits a  supersolution $\overline{w}_{\alpha_2,\beta_2} \in C^2(\mathbb{R}^n\setminus E_{\overline{r}})$. This supersolution is generalized symmetric with respect to $A$ and has the asymptotic behavior
 	\begin{equation}\label{omegasup}
 		\overline{w}_{\alpha_2,\beta_2}(x) = \frac{1}{2}x^{\mathsf{T}}Ax + c_{\alpha_2,\beta_2} + \begin{cases} 
 			O(r^{2-\min\{d(A,\epsilon_0), \beta\}}), & \text{if } d(A,\epsilon_0) \neq \beta, \\ 
 			O(r^{2-\beta} \ln r), & \text{if } d(A,\epsilon_0) = \beta, 
 		\end{cases}  
 	\end{equation}
 	as $|x| \to \infty$,	where $c_{\alpha_2,\beta_2}$ is a contant depending only on  $n, A, b, \theta, f, \alpha_2$ and $\beta_2$. If $\theta+f(x)>(n-2)\pi/2$,   the asymptotic behavior becomes
 	\begin{equation}\label{omegasupsc}
 		\overline{w}_{\alpha_2,\beta_2}(x) = \frac{1}{2}x^{\mathsf{T}}Ax + c_{\alpha_2,\beta_2} + \begin{cases} 
 			O(r^{2-\min\{d(A,0), \beta\}}), & \text{if } d(A,0) \neq \beta, \\ 
 			O(r^{2-\beta} \ln r), & \text{if } d(A,0) = \beta, 
 		\end{cases}  
 	\end{equation}
 	as $|x| \to \infty$. 
 	\end{prop}

We start by estimating $F(D^2\overline{w})$ from above with the following inequality.
\begin{lem}\label{spv}
Suppose $\epsilon_0>0$.	For any $C^2$ generalized symmetric function $\overline{w}$ with respect to $A$. Let $h=\overline{w}^{'}(r)/r$ and $\epsilon_0$ as in Remark \ref{MAcond}.  Suppose that $h > 0$ , $h' \geq 0$ in $[1,\infty)$,
		then the eigenvalues satisfy
	\begin{equation}\label{spvneq1}
		a_ih  \leq \lambda_i(D^2\overline{w}) \leq a_ih++ r^{-1}h'\sum_{j=1}^{n}a_j^2x_j^2, \quad \forall i=1,\cdots,n.
	\end{equation}
Moreover,	if 
\begin{equation}
	\label{limh}
	\lim_{r\to \infty} h=1,\quad \lim_{r\to \infty} rh^{'}=0,
\end{equation}
	then there exists $\overline{r}=\overline{r}\left(n,A,\epsilon_0,h,h^{'}\right)$  such that for any $r\geq\overline{r}$, 
	\begin{equation}\label{spvneq2}
		F(D^2\overline{w}) \leq g(\overline{a}_{\epsilon_0}),
	\end{equation}
 where $\overline{a}_{\epsilon_0} := \big(a_1h + a_nrh', a_2\left(h +\epsilon_0rh'\right), \cdots, a_n\left(h +\epsilon_0rh'\right)\big)$.
\end{lem}
\begin{proof}
The	estimate \eqref{spvneq1} is similar to Lemma \ref{evest}.	

	Let $V$ and $g$ as before. It's clear that
	\begin{equation}
		\label{spvneq3}
		0\leq a_1rh^{'}\leq V \leq a_nrh^{'}.
	\end{equation}
	
	 Recall the eigenvalues formular	\eqref{lmduw}, the difference between $g(\lambda(\overline{w}))$ and $g(\overline{a}_{\epsilon_0})$ gives
	\begin{equation}\label{ineqV1}
	g(\lambda(\overline{w}))-g(\overline{a}_{\epsilon_0})=\left(\theta_1V-a_nrh^{'}\right)\frac{\partial g}{\partial \lambda_1}(\tilde{a}_{\epsilon_0})+\sum_{i=2}^n\frac{\partial g}{\partial \lambda_i}(\tilde{a}_{\epsilon_0})\left(\theta_iV-\epsilon_0a_irh^{'}\right)
	\end{equation}
	where $\tilde{a}_{\epsilon_0}$ is a point lying in the segment between $\lambda(\overline{w})$ and $\overline{a}_{\epsilon_0}$.
	
By \eqref{limh}, one obtains
\[\tilde{a}_{\epsilon_0}\to a \quad \text{and} \quad \nabla g(\tilde{a}_{\epsilon_0})\to \nabla g(a) \ \text{as}\ r\to +\infty. \]
  That is, for sufficiently small $\epsilon$, there exists $\overline{r}$ such that
  \[\left|\nabla g(\tilde{a}_{\epsilon_0})- \nabla g(a)\right|<\epsilon,\quad \forall r\geq \overline{r}. \]
  Notice that $a_1\leq a_i$, $i=1,\cdots,n$, then by \eqref{spvneq3} and  \eqref{ineqV1}, the inequality
\begin{align*}
	g(\lambda(\overline{w}))-g(\overline{a}_{\epsilon_0})&\leq \left(V-a_nrh^{'}\right)\frac{\partial g}{\partial \lambda_1}(a)-\epsilon\left(\theta_1v-a_nrh^{'}\right)\\
	&\quad +\epsilon V\sum_{i=2}^n\theta_i-\epsilon_0\sum_{i=2}^n\frac{\partial g}{\partial \lambda_i}(\tilde{a}_{\epsilon_0})a_irh^{'}\\
	&\leq \left[\left(1-2\theta_1\right)V+a_nrh^{'}\right]\epsilon-\epsilon_0 rh^{'}\left[\sum_{i=2}^n\left(\frac{\partial g}{\partial \lambda_i}(a)-\epsilon\right)a_i\right] \\
	&\leq rh^{'}\left[2\epsilon a_n+\epsilon\epsilon_0\sum_{i=2}^na_i-\epsilon_0\sum_{i=2}^n\frac{\partial g}{\partial \lambda_i}(a)a_i\right] \\
	&= rh^{'}\left(2\epsilon a_n+\epsilon\epsilon_0\sum_{i=2}^na_i-\epsilon_0\sum_{i=2}^n\frac{a_i}{1+a_i^2}\right)\leq 0,
\end{align*}
holds when $0<\epsilon<\epsilon_0\left(2a_n+\epsilon_0\sum_{i=2}^n a_i\right)^{-1}\sum_{i=2}^n\frac{a_i}{1+a_i^2}$. 
\end{proof}

In order to find supersolutions of \eqref{Lu3.1}, Lemma \ref{spv} suggests us to consider the following second-order ODE:
\begin{equation}\label{spe}
	g\left(a_1h + a_nrh', a_2\left(h +\epsilon_0rh'\right), \cdots, a_n\left(h +\epsilon_0rh'\right)\right)=\underline{f}, \quad r>1
\end{equation}
with the initial data
\begin{equation}\label{spid}
	h(1)=\alpha_2,\quad h'(1)=\beta_2.
\end{equation}
We will show the global existence of solutions to problems \ref{spe} and \ref{spid} and determine their asymptotic behavior at infinity, which parallel to  Lemmas \ref{Ipthm}--\ref{csth}  in Sec. \ref{gssub}.

Since the proof method is similar, we proceed directly to the statement of the relevant lemma and corollary.

\begin{lem}\label{Jpthm}
Let $r\in [1,+\infty)$ and $h_2\in[\beta_2,\overline{h}(1)]$ with $h_2>0$. Then	there exists a unique positive smooth function $J(r;h_2)$ satisfying 
	\begin{equation}\label{imJeq}
		g(a_1h_2+a_nJ,a_2h_2+a_{2}\epsilon_0J,\cdots,a_nh_2+a_{n}\epsilon_0J)=\underline{f}(r).
	\end{equation}
	Especially, $J(r,h_2)=0$ if and only if $h_2(r)=\overline{h}(r)$,  and $J(r,h_2)$  is monotone increasing with respect to $r$ and monotone decreasing with respect to $h_2$. Furthermore, there exists $C>0$ such that
	\begin{equation}\label{eqJ}
		|J(r,h_2)-J(\infty,h_2)|\leq Cr^{-\beta}\  \text{and}\ \frac{\partial J}{\partial h_2}(\infty,1)=-d(A,\epsilon_0).
	\end{equation}
\end{lem}

\begin{rem}\label{rmsup}
	If $\theta+f(x)>(n-2)\pi/2$, let $\epsilon_0=0$ and replace $a_{\epsilon_0}$, $h_2$ and $J(r,h_2)$  with $\overline{a}_{J}$, $w$ and $H(r,w)$ respectively.
	 (see \cite[Lemma 11, 12]{BLW2024}). In this case, the right hand of the second inequality in  \eqref{eqJ}  becomes $-d(A,0)$.
\end{rem}

\begin{cor}\label{supst}
	Let $\overline{w}, h_2, A$ and $J$ as in Lemma \eqref{Jpthm}. If $h_2\in C^{1} [1,+\infty)$ satisfying
	\[0<\beta_2\leq h_2<\overline{h}(1),\quad h_2^{\prime}\geq 0,\quad   h_2^{\prime}\leq \frac{J(r,h_2)}{r}\ \text{in}\ (1,\infty),\]
and \eqref{limh}.	Then, $\overline{w} \in C^{2}(\mathbb{R}^n\setminus E_{\overline{r}})$ is a supersolution to \eqref{Lu3.1}.
\end{cor}

Similar to Lemma \ref{csth}, we can construct supersolution satisfied the assumption in Corollary \ref{supst}.
\begin{lem}\label{cspth}
	For any $0<\beta_2\leq  \overline{h}(1)$,	there exists a unique  solution $h(r;\beta_2)$ to
	\begin{equation}\label{dhsup}
		\left.\left\{\begin{array}{ll}
			\frac{dh_2}{dr}=\frac{J(r,h_2)}{r},& r>1,\\
			h(1)=\beta_2.
		\end{array}\right.
		\right.
	\end{equation}
	Futhermore, 
	\begin{enumerate}
		\item[(i)] $\beta_2\leq h_2(r;\beta_2)\leq 1$, $\partial h_2/\partial r\geq 0$. More specifically, $h_2(r;\overline{h}(1))=\overline{h}(1)$ and $\beta_2<h_2(r;\beta_2)< 1$, $\forall r>1, \beta_2<\underline{h}(\overline{r})$.
		\item[(ii)]  $h_2$ has asymptotic
		behavior \begin{equation}\label{absuph}
			h_2(r) = 1 + 
			\begin{cases} 
				O(r^{-\min\{d(A,\epsilon_0), \beta\}}), & \text{if } d(A,\epsilon_0) \neq \beta, \\ 
				O(r^{-\beta} \ln r), & \text{if } d(A,\epsilon_0) = \beta, 
			\end{cases}
		\end{equation}
	 and $rh_2^{'}\to 0$	as $r\to \infty$, where $d(A,\epsilon_0)$ is defined in  \eqref{MAeps}.  Moreover, if $\theta+f(x)>(n-2)\pi/2$,  Remark \ref{rmsup} implies  
	 $h_2$ has asymptotic
	 behavior \begin{equation}\label{absuphsc}
	 	h_2(r) = 1 + 
	 	\begin{cases} 
	 		O(r^{-\min\{d(A,0), \beta\}}), & \text{if } d(A,0) \neq \beta, \\ 
	 		O(r^{-\beta} \ln r), & \text{if } d(A,0) = \beta, 
	 	\end{cases}
	 \end{equation}
	and $rh_2^{'}\to 0$ as $r\to \infty$. 
		\item[(iii)] $h_2(r;\beta_2)$ is continuous and strictly increasing with respect to $\beta_2$.
	\end{enumerate}
\end{lem}

\noindent\textbf{Proof of Propsition \ref{suplem}.} For any $\alpha_2\in \mathbb{R}$, $0<\beta_2<\overline{h}(1)$ and $0<\beta_2<\overline{h}(1)$,  let
\begin{equation}\label{supw}
	\begin{aligned}
		\overline{w}_{\alpha_2,\beta_2}(x)=\phi_{_{\alpha_2,\beta_2}}(r)&=\alpha_2+\int_{1}^rsh_2(s;\beta_2)ds\\
		&=:\alpha_2+\frac{1}{2}\sum_{i=1}^na_ix_i^2+\mu_2(\beta_2)-\int_{r}^{\infty}s(h_2(s;\beta_2)-1)ds,
	\end{aligned}
\end{equation}
where 
\[\mu_2(\beta_2)=\int_{1}^{\infty}s(h_2(s;\beta_2)-1)ds-\frac{1}{2}.\]

Combine the above results with Corollary \ref{supst}, $\overline{w}$ is a  supersolution of \eqref{Lu3.1} in $\mathbb{R}^n\setminus E_{\overline{r}}$, where $\overline{r}$ s a constant $n,A,b,\epsilon_0,\beta$ and $\delta$.  $\overline{w}$ has  initial date $\overline{w}_{\alpha_2,\beta_2}(1)=\alpha_2$ and $\overline{w}_{\alpha_2,\beta_2}^{'}(1)=\beta_2$.

By \eqref{absuph}, it's easy to see when $r\to \infty$, 
\begin{equation*}
	\int_{r}^{\infty}s(h_2(s;\beta_2)-1)ds=
	\begin{cases} 
		O(r^{2-\min\{d(A,\epsilon_0), \beta\}}), & \text{if } d(A,\epsilon_0) \neq \beta, \\ 
		O(r^{2-\beta} \ln r), & \text{if } d(A,\epsilon_0) = \beta, 
	\end{cases}
\end{equation*}
In view of  Remark \ref{rmsup} and \eqref{absuphsc},  if $\theta+f(x)>(n-2)\pi/2$,   
 \begin{equation*}
\int_{r}^{\infty}s(h_2(s;\beta_2)-1)ds=
	\begin{cases} 
		O(r^{2-\min\{d(A,0), \beta\}}), & \text{if } d(A,0) \neq \beta, \\ 
		O(r^{2-\beta} \ln r), & \text{if } d(A,0) = \beta, 
	\end{cases}
\end{equation*}
as $r\to \infty$.
Inserting the above equality into \eqref{supw}, we can obtain the asymptotic behavior \eqref{omegasup} and \eqref{omegasupsc}.
 \qed

\subsection{Radial supersolutions outside $\Omega$}\label{Rsup}

 By our assumption, $B_{\frac{1}{\sqrt{a_n}}}\subset E_1 \subset \subset \Omega$.
Let $\tilde{a}:=\tan \frac{\theta-\delta}{n}>0$, it's easy to see
\[g\left(\tilde{a},\cdots,\tilde{a}\right)=\theta-\delta.\]
 Denote $r=\sqrt{\tilde{a}}|x|, x\in \mathbb{R}^n$ and $\tilde{w}=\tilde{w}(r)$. Let $\tilde{w}$ satisfies 
\begin{equation}\label{rsupode}
	\begin{cases}
		F(D^2\tilde{w})=g(\tilde{a}h+\tilde{a}rh^{'},\tilde{a}h,\cdots,\tilde{a}h)=\theta-\delta,& r>\sqrt{\frac{\tilde{a}}{a_n}},\\
		\tilde{w}\left(\sqrt{\frac{\tilde{a}}{a_n}}\right)=\alpha_3, \quad   \tilde{w}^{'}\left(\sqrt{\frac{\tilde{a}}{a_n}}\right)=\beta_3.
	\end{cases}
\end{equation}
Here, $h=\tilde{w}^{'}(r)/r$. It's clear that $\tilde{w}$ is a supersolution of \eqref{eq}   in $\mathbb{R}^n\setminus \Omega$.

\begin{prop}
	 Let \( \alpha_3 \in \mathbb{R} \) and \( \beta_3 > \sqrt{\frac{\tilde{a}}{a_n}} \). If $\theta+f(x)>(n-2)\pi/2$, problem \eqref{rsupode} admits a unique solution \( h_{3}(r) \).
	Moreover, for any \( r \geq \sqrt{\frac{\tilde{a}}{a_n}} \), \( h_{3}(r) > 1 \) and \(h_{3}^{'}(r) < 0 \).
\end{prop}
\begin{proof}
	This propsition can be proved in the same way seeking a smooth solution analogous to Lemmas \ref{Ipthm}--\ref{csth}. Now, we  just deal with the simple version: $\overline{f}=\theta-\delta>(n-2)\pi/2$ and  $\underline{h}=1$.
	
	Indeed, as argued for Lemma \ref{Ipthm}, we observe that for each $1<h<\gamma_3:=\beta_3\sqrt{\frac{a_n}{\tilde{a}}}$ there
	exists a nonpositive smooth function $I(h_3)$ such that 
	\begin{equation}\label{Isupeq}
		g\left(\tilde{a}h_3 + \tilde{a}I, \tilde{a}h_3, \cdots, \tilde{a}h_3 \right)=\theta-\delta.
	\end{equation}
	Moreover, $I(h_3)=0$ if and only if $h_3\equiv1$,  $I(h_3)$ is monotone decreasing in $h_3$ and 
	\[\frac{d I}{d h_3}(1)=-n.\] 
	
	It's clear that $\tilde{w}$ is a supersolution of \eqref{eq}   in $\mathbb{R}^n\setminus \Omega$ if $h_3$ satisfies equation \eqref{Isupeq}.
	In order to obtain a solution of \eqref{rsupode}, this leads us to study
	\begin{equation}\label{hrsupode}
		\begin{cases}
			\frac{d h_3}{dr}=\frac{I(h_3)}{r},& r>\sqrt{\frac{\tilde{a}}{a_n}},\\
		h_3\left(\sqrt{\frac{\tilde{a}}{a_n}}\right)=\gamma_3.  
		\end{cases}
	\end{equation}
	Similar to the result in Lemma \ref{csth}, we easily infer that problem \eqref{hrsupode} admits a unique smooth solution $h_3(r;\gamma_3)$ defined on $\left[\sqrt{\frac{\tilde{a}}{a_n}},+\infty\right)$, such that for $r>\sqrt{\frac{\tilde{a}}{a_n}}$,
	\[h_3(r;\gamma_3)>1 \quad \text{and}\quad h_3^{'}(r;\gamma_3)<0.\] 
	Moreover, $h_3(r;\gamma_3)\equiv1$ if and only if $\gamma_3=1$,  $h_3(r;\gamma_3)$ is strictly increasing with respect to $\gamma_3$  and $\lim\limits_{\gamma_3\to +\infty} h_3(r;\gamma_3)=+\infty, \ \forall r\geq \sqrt{\frac{\tilde{a}}{a_n}}$.
\end{proof}

Let 
\begin{equation}\label{suptw}
	\begin{aligned}
		\tilde{w}_{\alpha_3,\beta_3}(x)=\phi_{_{\alpha_3,\beta_3}}(r)&=\alpha_3+\int_{\sqrt{\frac{\tilde{a}}{a_n}}}^rsh_3(s;\gamma_3)ds\\
		&=:\alpha_3+\frac{\tilde{a}}{2}|x|^2+\mu_3(\gamma_3)-\int_{r}^{\infty}s(h_3(s;\gamma_3)-1)ds\\
		&=\frac{\tilde{a}}{2}|x|^2+\alpha_3+\mu_3(\gamma_3)+O\left(|x|^{2-n}\right)\quad \text{as }\ |x|\to +\infty,
	\end{aligned}
\end{equation}
where 
\[\mu_3(\gamma_3)=\int_{\sqrt{\frac{\tilde{a}}{a_n}}}^{\infty}s(h_3(s;\gamma_3)-1)ds-\frac{\tilde{a}}{2a_n}.\]

Clearly, $\tilde{w}_{\alpha_3,\beta_3}$ is a  solution of \eqref{rsupode} on $\left[\sqrt{\frac{\tilde{a}}{a_n}},+\infty\right)$ with initial date $\tilde{w}_{\alpha_3,\beta_3}(\sqrt{\frac{\tilde{a}}{a_n}})=\alpha_3$ and $\tilde{w}_{\alpha_3,\beta_3}^{'}(\sqrt{\frac{\tilde{a}}{a_n}})=\beta_3$. Futhermore, $\mu(\gamma_3)$ is strictly increasing in $\gamma_3$ and $\mu_3(\gamma_3)\to +\infty$ as $\gamma_3\to +\infty$.

\section{Proof of Theorem \ref{tedp}}\label{sec5}
In this section, we prove Theorem \ref{tedp} by applying an adapted Perron’s method (see \cite{CL2003,BLL2014,BLZ2015,LZ2019,LW2024}). The main ingredient of the proof is to demonstrate the existence
of a viscosity subsolution $\underline{u}$ of \eqref{eedp} with prescribed Dirichlet boundary value and
asymptotic behavior at infinity, and also a viscosity supersolution $\overline{u}\geq \underline{u}$ but agreeingon $\underline{u}$ at infinity.

\vskip 5mm
\noindent \textbf{Proof of Theorem \ref{tedp}.}  Analogous to the Lemma 4.3  in \cite{LZ2019} and \cite[Pages 14--15]{WB2024arxiv}, we only to  prove in the case that the matrix $A$
is diagonal and the vector $b$ vanishes.

Without loss of generality, we also assume $E_1 \subset \subset\Omega \subset \subset E_{r_1} \subset \subset E_{r_2}$, where $r_1 \leq r_2<\overline{r}$ are fixed constants, $\overline{r}$ is defined in Propsition \ref{suplem} and Lemma \ref{spv}.  We next split the proof into three steps.

\textbf{Step 1.} Construct a viscosity subsolution \(\underline{u}\) of \eqref{eedp} (or \eqref{eedp1}) with \(\underline{u} = \varphi\) on \(\partial \Omega\). Moreover, for $f\equiv 0$ and $\delta<\theta<n\pi/2-\delta$, \(\underline{u}\)  has  the asymptotics behavior 
\begin{equation}\label{absube}
	\underline{u}(x)=\frac{1}{2}\sum_{i=1}^na_ix_i^2 + c+O(|x|^{2-d(A,\epsilon_0)}), \quad \text{as} \ |x|\to +\infty.
\end{equation}
For $f\not\equiv 0$ and $(n-2)\pi/2<\theta<n\pi/2$, if $\beta \neq d(A,0)$,  \(\underline{u}\)  has  the asymptotics behavior 
\begin{equation}\label{absube1}
		\underline{u}(x)=\frac{1}{2}\sum_{i=1}^na_ix_i^2 + c+O(|x|^{2-\min \{d(A,0),\beta\}}), \quad \text{as} \ |x|\to +\infty.
\end{equation}
If $\beta = d(A,0)$, 
\(\underline{u}\)  has  the asymptotics behavior 
\begin{equation}\label{absube2}
	\underline{u}(x)=\frac{1}{2}\sum_{i=1}^na_ix_i^2 + c+O(\ln|x||x|^{2-d(A,0)}), \quad \text{as} \ |x|\to +\infty.
\end{equation}

Let $A\in \mathcal{A}_{0}$ and \( K > 0 \)  large enough such that the function \( Q_{\xi} \) given by Lemma \ref{bdsublem} satisfies
\begin{equation}\label{bdsubeq}
	F(\lambda(D^2Q_{\xi})) = g(K\lambda(A)) = \theta+\delta\geq \theta+f(x),\quad \forall x\in \mathbb{R}^n\setminus \Omega.
\end{equation}
Hence, \( Q_{\xi} \) is a smooth subsolution of \eqref{eedp}.

 For \( x \in \mathbb{R}^n \setminus \Omega \), set
\[
Q(x) = \max\{Q_{\xi}(x) : \xi \in \partial \Omega\}.
\]
Then \( Q \) is a viscosity subsolution of \eqref{eedp} by \cite[Lemma 4.2]{CIL1992} and \( Q = \varphi \) on \( \partial \Omega \) by Lemma \ref{bdsublem}.

Let
\[\alpha_1 = \min\{ Q_{\xi}(x) : \xi \in \partial \Omega, \; x \in \overline{E}_{r_1}\setminus \Omega \}.\]
and
\[u_{\alpha_1,\beta_1}(x)=\alpha_1+\int_{r_1}^r\tau h_1(\tau;\beta_1)d\tau, \quad \forall r\geq 1,\ \beta_1>\underline{h}(1),\]
where $h_1(r;\beta_1)$ is given in Lemma \ref{csth}.

By Propsition \ref{sublem}, $u_{\alpha_1,\beta_1}$ is a $C^2$  subsolution of \eqref{eq} in $\mathbb{R}^n\setminus \Omega$.

It's clear that $u_{\alpha_1,\beta_1}=\alpha_1$ on $\partial E_{r_1}$ and $\max_{\partial E_{r_1}}u_{\alpha_1,\beta_1}\leq \min_{\partial E_{r_1}}Q$  by defination. Indeed, we have $u_{\alpha_1,\beta_1}\leq \alpha_1\leq Q$ on $\overline{E}_{r_1}\setminus \Omega$. By Lemma \ref{bdsublem} and the defination of $Q$, we also get $u_{\alpha_1,\beta_1}\leq \alpha_1\leq \varphi$ on $\partial \Omega$.

It follows  Lemma \ref{csth} that  $u_{\alpha_1,\beta_1}$ is strictly increasing with respect to $\beta_1$ and 
\begin{equation}\label{ua1b1}
	\lim_{\beta_1\to +\infty} u_{\alpha_1,\beta_1}(x)=+\infty,\quad \forall r\geq1.
\end{equation}
As showed in \eqref{omega1}  and \eqref{w}, for any $ \beta_1 \geq  \underline{h}(1)$, we have
\begin{align*}
	u_{\alpha_1,\beta_1}(x)&=\alpha_1+\int_{r_1}^r\tau h_1(\tau;\beta_1)d\tau\\
	&=\alpha_1+\frac{r^2-r_1^2}{2}+\int_{r_1}^r\tau \left(h_1(\tau;\beta_1)-1\right)d\tau\\
	&=\frac{1}{2}\sum_{i=1}^na_ix_i^2+\left(\alpha_1-\frac{1}{2}r_1^2+\int_{r_1}^\infty\tau \left(h_1(\tau;\beta_1)-1\right)d\tau\right)-\int_{r}^\infty\tau \left(h_1(\tau;\beta_1)-1\right)d\tau\\
	&=:\frac{1}{2}\sum_{i=1}^na_ix_i^2+\mu_1(\beta_1)-\int_{r}^\infty\tau \left(h_1(\tau;\beta_1)-1\right)d\tau,
\end{align*}
where 
\[\mu_1(\beta_1):=\alpha_1-\frac{1}{2}r_1^2+\int_{r_1}^\infty\tau \left(h_1(\tau;\beta_1)-1\right)d\tau.\]

For fixed $r_2>r_1$, in light of \eqref{ua1b1}, there exists $\hat{\beta}_1>\underline{h}(1)$ large enough such that
\begin{equation}
	\label{uqieq}
	\min_{\partial E_{r_2}} 	u_{\alpha_1,\hat{\beta}_1}>\max_{\partial E_{r_2}}Q.
\end{equation}

Let $c_{*}:= \mu_1(\hat{\beta}) $. Via Lemma \ref{csth}, 
 for each $c>c_{*}$, there exists a unique $\beta_1(c)>\hat{\beta}_1$ such that $\mu_1(\beta_1(c))=c$. 

  Using   Lemma \ref{csth} again, for $f\equiv 0$ and $\delta<\theta<n\pi/2-\delta$,    let $\beta$ large enough such that $\beta>n\geq d(A,\epsilon_0)$, then we get asymptotic behaviors \eqref{absube}--\eqref{absube2}.

For every $c>c_{*}$, define
\begin{equation}
	\label{subu}
	\underline{u}(x) =
	\begin{cases} 
		Q(x), & x \in E_{r_1} \setminus \Omega, \\
		\max\{Q(x), u_{\alpha_1, \beta_1(c)}(x)\}, & x \in E_{r_2} \setminus E_{r_1}, \\
		u_{\alpha_1, \beta_1(c)}(x), & x \in \mathbb{R}^n \setminus E_{r_2}.
	\end{cases}
\end{equation}
From \cite[Lemma 4.2]{CIL1992}, we deduce that $\underline{u}$ is a viscosity subsolution
of \eqref{eedp} (or \eqref{eedp1}) satisfying  $\underline{u}=\varphi$ on $\partial \Omega$ and asympotoic behaviors \eqref{absube}--\eqref{absube2}.

\textbf{Step 2.} Construct a viscosity supersolution \( \overline{u} \) of \eqref{eedp} (or \eqref{eedp1}) to satisfy
\[
\underline{u} < \overline{u} \text{ in } \mathbb{R}^n \setminus \Omega \quad \text{and} \quad \lim_{|x| \to \infty} (\overline{u} - \underline{u})(x) = 0.
\]

%

We first consider the supersolution for problem \ref{eedp}.    Recall that \(\overline{w}_{\alpha_2,\beta_2}\) and \(\tilde{w}_{\alpha_3,\beta_3}\)  introduced in Sections~\ref{Rsup} and~\ref{gssup}, respectively. We construct a supersolution of \eqref{eedp} in \(\mathbb{R}^n \setminus \Omega\) by splicing them together and choosing appropriate parameters \(\alpha_3\) and \(\beta_3\) such that
\begin{equation}\label{w23}
	\max_{\partial B_{r_3}} \tilde{w}_{\alpha_3,\beta_3} \leq \min_{\partial B_{r_3}} \overline{w}_{\alpha_2,\beta_2} \quad \text{and} \quad \min_{\partial B_{r_4}} \tilde{w}_{\alpha_3,\beta_3} \geq \max_{\partial B_{r_4}} \overline{w}_{\alpha_2,\beta_2},
\end{equation}
where \(r_3\) and \(r_4\) are two fixed numbers such that \(E_{r_1} \subset E_{r_2}\subset E_{\overline{r}}\subset B_{r_3}  \subset B_{r_4}\).

Let
\[
M(\alpha_2): = \min\{\overline{w}_{\alpha_2,\beta_2}(x) : x \in \overline{B_{r_3}} \setminus E_{\overline{r}}\}
\]
and
\[\hat{\alpha}_3=M(\alpha_2)-\int_{\sqrt{\frac{\tilde{a}}{a_n}}}^{\sqrt{\tilde{a}}r_3}sh_3(\tau;\gamma_3)d\tau.\]

Therefore, $\tilde{w}_{\hat{\alpha}_3,\beta_3}$ becomes
\begin{equation}\label{nw3}
	\tilde{w}_{\hat{\alpha}_3,\beta_3}(x)=M(\alpha_2)+\int_{\sqrt{\tilde{a}}r_3}^{\sqrt{\tilde{a}}|x|}sh_3(\tau;\gamma_3)d\tau.
\end{equation}
It's clear that $	\tilde{w}_{\hat{\alpha}_3,\beta_3}$ satisfies the first ineuality in \eqref{w23}.

 In the following, we show the second one holds when $\gamma_3$ (or $\beta_3$) is sufficiently large.
Notice that 
\[\max_{\partial B_{r_4}} \overline{w}_{\alpha_2,\beta_2}=\alpha_2+\max_{\partial B_{r_4}}\int_{r_1}^r \tau h_2(\tau;\gamma_2)d\tau\]
and 
\begin{align*}
	\min_{\partial B_{r_4}} \tilde{w}_{\hat{\alpha}_3,\beta_3}&=M+\int_{\sqrt{\frac{\tilde{a}}{a_n}}}^{\sqrt{\tilde{a}}r_4}sh_3(\tau;\gamma_3)d\tau\\
	&=\alpha_2+\min \left\{\int_{r_1}^r \tau h_2(\tau;\gamma_2)d\tau: x \in \overline{B_{r_3}} \setminus E_{\overline{r}} \right\}+\int_{\sqrt{\frac{\tilde{a}}{a_n}}}^{\sqrt{\tilde{a}}r_4}sh_3(\tau;\gamma_3)d\tau.
\end{align*}
Hence, by the monotonicity of $\tilde{w}_{\hat{\alpha}_3,\beta_3}$ with respect to $\beta_3$, one can infer that if $\beta_3$ is
sufficiently large (fix a large $\beta_3=\hat{\beta}_3$), then the second inequality in \eqref{w23} holds.

In order to construct a  supersolution $\overline{u}$ satisfies $\overline{u}\geq \varphi$ on $\partial \Omega$, our strategy is to let $\tilde{w}_{\hat{\alpha}_3,\hat{\beta}_3}$ be a   supersolution near  $\partial \Omega$ and let  $\overline{w}_{\alpha_2,\beta_2}$ be a supersolution near infinity.

For any $x\in \partial \Omega$, since $M(\alpha_2)>\alpha_2$ by \eqref{supw} and $B_{1/\sqrt{a_n}}\subset \Omega$, we get
\begin{align*}
	\overline{u}&= \tilde{w}_{\hat{\alpha}_3,\hat{\beta}_3}\geq M(\alpha_2)+\int_{\sqrt{\tilde{a}}r_3}^{\sqrt{\frac{\tilde{a}}{a_n}}}sh_3(\tau;\gamma_3)d\tau\geq  \alpha_2+\int_{\sqrt{\tilde{a}}r_3}^{\sqrt{\frac{\tilde{a}}{a_n}}}sh_3(\tau;\gamma_3)d\tau.
\end{align*}
 
 Let $\alpha_2(c)=c-\mu_2(\beta_2)$, then for $c>c_{**}:=\max_{\partial \Omega} \varphi-\int_{\sqrt{\tilde{a}}r_3}^{\sqrt{\frac{\tilde{a}}{a_n}}}sh_3(\tau;\gamma_3)d\tau+\mu_2(\beta_2)$, we have
 \[	\overline{u}(x)\geq c_{**}-\mu(\beta_2)+\int_{\sqrt{\tilde{a}}r_3}^{\sqrt{\frac{\tilde{a}}{a_n}}}sh_3(\tau;\gamma_3)d\tau\geq \max_{\partial \Omega} \varphi,\quad \forall x\in \partial\Omega.\]

 For every $c>\max\left\{c_{*},c_{**}\right\}$, define
\begin{equation}
	\label{supu}
	\overline{u}(x) =
	\begin{cases} 
		\tilde{w}_{\hat{\alpha}_3,\hat{\beta}_3}(x), & x \in B_{r_3} \setminus \Omega, \\
		\min\{	\tilde{w}_{\hat{\alpha}_3,\hat{\beta}_3}(x), \overline{w}_{\alpha_2,\beta_2}(x)\}, & x \in B_{r_4} \setminus B_{r_3}, \\
	\overline{w}_{\alpha_2(c),\beta_2}(x), & x \in \mathbb{R}^n \setminus B_{r_4}.
	\end{cases}
\end{equation}
It's clear $\overline{u}$ satisfies $\overline{u}\geq \varphi$ on $\partial \Omega$.  In view of  Lemma \ref{cspth}, it's easy to see 
\begin{equation}
	\label{absupe}
	\lim\limits_{|x|\to +\infty}|\overline{u}(x)-\underline{u}(x)|=0.
\end{equation}

To use the Perron’s method, now we need only to prove that
\[\underline{u}\leq \overline{u}\quad \text{on}\ \mathbb{R}^n\setminus \Omega.\]

In view of \eqref{w} and \eqref{supw}, notice that $h_1(r;\beta_1)>1$ and $h_2(r;\beta_2)<1$, we get 
\begin{equation}\label{u1w2cp}
	u_{\alpha_1, \beta_1(c)}(x)\leq \frac{1}{2}\sum_{i=1}^na_ix_i^2 + c\leq \overline{w}_{\alpha_2,\beta_2}(x)\quad \text{on}\ \mathbb{R}^n\setminus E_{\overline{r}}.
\end{equation}

Next, we show $Q_{\xi}\leq \tilde{w}_{\hat{\alpha}_3,\hat{\beta}_3}$ on $\overline{B}_{r_4}\setminus \Omega$ and hence $Q_{\xi}\leq \tilde{w}_{\hat{\alpha}_3,\hat{\beta}_3}$ on $\overline{E}_{r_2}\setminus \Omega$.

For every $\xi\in \partial\Omega$, we have
\[Q_{\xi}<\varphi\leq \tilde{w}_{\hat{\alpha}_3,\hat{\beta}_3}.\]

On the other hand, for any  $x\in \partial {B}_{r_4}$, by Lemma  \ref{bdsublem}, there exists $\overline{c}=\overline{c}(\Omega,K,A,C)$ such that 
\begin{align*}
	\omega_{\xi}(x)\leq \frac{K}{2}x^{\mathsf{T}}Ax+\overline{c}\leq C(\Omega,K,A,C,r_4)\leq \tilde{w}_{\tilde{\alpha}_3,\hat{\beta}_3},
\end{align*}
provided $\hat{\beta}_3$ is sufficiently large.

In view of \eqref{rsupode} and \eqref{bdsubeq}, 
\[\sum_{i=1}^n\arctan \lambda_i(D^2\omega_{\xi})=\theta+\delta> \theta-\delta=\sum_{i=1}^n\arctan \lambda_i(D^2\tilde{w}_{\tilde{\alpha}_3,\tilde{\beta}_3}) \quad \text{in}\ B_{r_4}\setminus \overline{\Omega}.\]
We deduce the result from comparison principle. Hence, $Q\leq \tilde{w}_{\tilde{\alpha}_3,\tilde{\beta}_3}$ on $\overline{E}_{r_2}\setminus \Omega$.

 Finally, we claim $u_{\alpha_1, \beta_1(c)}\leq \tilde{w}_{\tilde{\alpha}_3,\hat{\beta}_3}$ on $\overline{B}_{r_4}\setminus E_{r_1}$. 

By \eqref{w23} and \eqref{u1w2cp}, it's clear the inequality holds on $\partial B_{r_4}$. On $\partial E_{r_1}$, we also let $\hat{\beta}_3$ large enough such that the inequality holds. Then by comparison principle, we get the desired result.

In summary, when  $f\not\equiv 0$ and $(n-2)\pi/2<\theta<n\pi/2$, by the defination of $\overline{u}$, $\underline{u}$, we conclude that $\underline{u}\leq \overline{u}$ on $\mathbb{R}^n\setminus \Omega$.

On the other hand, when $f\equiv 0$, we can take $\overline{u} = \frac{1}{2}\sum_{i=1}^na_ix_i^2 + c$ directly as the desired supersolution, where
\[c>\max\left\{c_{*},\max \left\{\frac{K-1}{2}x^{\mathsf{T}}Ax\vert \ x\in\partial\Omega  \right\}+\overline{c}  \right\}.\]
It's easy to verify that   $\overline{u}\geq \varphi$ on $\partial \Omega$ by Lemma \ref{bdsublem} and $\overline{u}\geq u_{\alpha_1, \beta_1(c)}$ on $\partial E_{r_2}$ by \eqref{uqieq}  and \eqref{u1w2cp}.
Consequently, comparison principle gives $Q\leq \overline{u}$ on $\overline{E}_{r_2}\setminus \Omega$.  Combined with \eqref{u1w2cp}, we also have $\underline{u}\leq \overline{u}$ on $\mathbb{R}^n\setminus \Omega$.

\textbf{Step 3.}    Construct a viscosity solution \( u \) to problem \eqref{eedp}.  
With \( \underline{u} \) and \( \overline{u} \) above. 
 
If $f\equiv 0$ and $\delta<\theta<n\pi/2-\delta$, define  
 \begin{align*}
 	u(x) &:= \sup\{v(x) \mid v \in C^0(\mathbb{R}^n\setminus\Omega), \; F(D^2v) \geq \theta \text{ in } \mathbb{R}^n\setminus \overline{\Omega} \\
 	& \quad  \text{ in the viscosity sense}, \; \underline{u} \leq v \leq \overline{u} \text{ in } \mathbb{R}^n\setminus \Omega, \; v = \varphi \text{ on } \partial \Omega\}.
 \end{align*}
By Lemma \ref{pemd}, we conclude that $u \in C^0(\mathbb{R}^n\setminus\Omega)$ is the unique viscosity solution
 to problem \eqref{eedp1}.  Thanks to \eqref{absube} and \eqref{absupe}, $u$ has the asymptotic behavior of probblem \eqref{eedp1}.

If $f\not\equiv 0$ and $(n-2)\pi/2<\theta<n\pi/2$,  Perron’s method as in Lemma \ref{pemd} could not be directly adapted to the problem for \eqref{eedp} (see \cite{LW2024}). We will provide a new proof by using the local solvability Theorem \ref{lsdp} for  supercritical phase for the Dirichlet problem.

 For $c>\max \{c_{*},c_{**}\}$, let $\mathcal{S}_c$ denote the set of $v\in C^0(\mathbb{R}^n \setminus \Omega )$ which are viscosity subsolutions of
 \eqref{eedp} in $\mathbb{R}^n \setminus \overline{\Omega}$ satisfying  
\begin{align*}
 v = \varphi \ \text{ on } \ \partial \Omega
\end{align*}
and
\[\underline{u}\leq v\leq \overline{u}\ \text{ in } \ \mathbb{R}^n\setminus \Omega.\]

Therefore, $\underline{u}\in \mathcal{S}_c$. Let
\[u(x)=\sup \left\{v(x)\vert \ v(x)\in \mathcal{S}_c \right\}.\] 
Then $u$ is of class $C^0(\mathbb{R}^n\setminus \overline{\Omega})$.
Thanks to \eqref{absube1}--\eqref{absube2} and \eqref{absupe}, $u$ has the asymptotic behavior of probblem \eqref{eedp}.

Next, we prove that $u$ satisfies the boundary condition. It is obvious from the definition of $\underline{u}$ that
\[\liminf_{x\to \xi}u(x)  \geq \lim_{x\to \xi}\underline{u}(x)=\varphi(\xi),\quad \forall \xi \in \partial\Omega. \]

In view of \cite[Proposition B]{WB2024arxiv},  there exists  $\omega\in C^0(\overline{E}_{r_1+1}\setminus \Omega)$ satisfying
\[
\begin{cases}
	\Delta \omega = 0, & \text{in } E_{r_1+1} \setminus \overline{\Omega}, \\
	\omega = \varphi, & \text{on } \partial \Omega, \\
	\omega = \max_{\partial\overline{E}_{r_1+1}} \overline{u}, & \text{on } \partial E_{r_1+1}.
\end{cases}
\]
Since $\theta+f(x)>(n-2)\pi/2$,  the viscosity subsolution \(v\in \mathcal{S}_c\) satisfies \(\Delta v \geq 0\) in viscosity sense (see \cite[Lemma 2.1]{WY2014}). Therefore,  for every $v\in \mathcal{S}_c$, by \(v \leq \omega\) on \(\partial(\overline{E}_{r_1} \setminus \Omega)\), we have
\[
v \leq \omega \quad \text{in } E_{r_1} \setminus \overline{\Omega}.
\]
It follows that
\[
u \leq \omega \quad \text{in } E_{r_1} \setminus \overline{\Omega},
\]
and then
\[
\limsup_{x \to \xi} u(x) \leq \lim_{x \to \xi} \omega(x) = \varphi(\xi), \quad \forall \xi \in \partial \Omega.
\]

Finally, we prove \( u \) is a viscosity solution of \eqref{eq}. For any \(\overline{x} \in \mathbb{R}^n \setminus \overline{\Omega}\), fix some \(\epsilon > 0\) such that \( B_\epsilon(\overline{x}) \subset \mathbb{R}^n \setminus \overline{\Omega} \). By the definition of \( u \), \( u \leq \overline{u} \). By Theorem \ref{lsdp}, there is a viscosity solution \(\widetilde{u} \in C^0(\overline{B_\epsilon(\overline{x})})\) to
\[
\begin{cases}
	\sum_{i=1}^n\arctan(\lambda_i(D^2\widetilde{u})) = \theta+f(x), & x \in B_\epsilon(\overline{x}), \\
	\widetilde{u} = u, & x \in \partial B_\epsilon(\overline{x}). 
\end{cases}
\]
By the maximum principle in \cite{JLS1988}, $u \leq \widetilde{u}\leq \overline{u}$ in $B_\epsilon(\overline{x})$, since the right side of the equation is in supercritical phase. Define
\[
w(y) = 
\begin{cases} 
	\widetilde{u}(y), & \text{if } y \in B_\varepsilon(\overline{x}), \\
	u(y), & \text{if } y \in \mathbb{R}^n \setminus (\Omega \cup B_\varepsilon(\overline{x})).
\end{cases}
\]

Clearly, \( w \in S_c \). So, by the definition of \( u \), \( u \geq w \) on \( B_\varepsilon(\overline{x}) \). It follows that \( u \equiv \widetilde{u} \) in \( B_\varepsilon(\overline{x}) \). Therefore \( u \) is a viscosity solution of \eqref{eedp}.  \qed

\end{document}